\documentclass[a4paper]{article}
\usepackage{geometry}
\usepackage{amsthm,amssymb,amsmath,enumerate,graphicx,epsf,bm}
\usepackage{authblk}
\usepackage[numbers,sort&compress]{natbib}  
\usepackage[utf8x]{inputenc}
\usepackage[percent]{overpic}
\usepackage{epstopdf}
\usepackage{multicol}
\usepackage{multirow}

\usepackage{centernot}

\newtheorem{theorem}{Theorem}
\newtheorem{proposition}{Proposition}[section]

\newtheorem{lemma}[proposition]{Lemma}

\newcommand{\Lr}[1]{Lemma~\ref{#1}}
\newcommand{\Tr}[1]{Theorem~\ref{#1}}

\newcommand{\Prr}[1]{Pro\-position~\ref{#1}}

\newtheorem{remark}[proposition]{Remark}

\newcommand{\pint}{interface}

\newcommand{\labtequ}[2]{%\labtequc{#1}{#2}}
 \begin{equation} \label{#1} \begin{minipage}[c]{0.9\textwidth}  #2 \end{minipage} \ignorespacesafterend \end{equation} }

\newtheorem{question}[proposition]{Question}

\DeclareMathAlphabet{\pazocal}{OMS}{zplm}{m}{n}

\title{Self-avoiding walks and polygons on hyperbolic graphs}

\author{Christoforos Panagiotis}
\affil{Universit{\'e} de Gen{\`e}ve, Geneva, Switzerland}
	
\begin{document}
\date{}
\maketitle

\begin{abstract}
We prove that for the $d$-regular tessellations of the hyperbolic plane by $k$-gons, there are exponentially more self-avoiding walks of length $n$ than there are self-avoiding polygons of length $n$. We then prove that this property
implies that the self-avoiding walk is ballistic, even on an arbitrary vertex-transitive
graph. Moreover, for every fixed $k$, we show that the connective constant for self-avoiding walks satisfies the asymptotic expansion $d-1-O(1/d)$ as $d\to \infty$; on the other hand, the connective constant for self-avoiding polygons remains bounded. Finally, we show for all but two tessellations that the number of self-avoiding walks of length $n$ is comparable to the $n$th power of their connective constant. Some of these results were previously obtained by Madras and Wu \cite{MaWuSAW} for all but finitely many regular tessellations of the hyperbolic plane. 
\end{abstract}

\section{Introduction}

A \textit{self-avoiding walk} (abbreviated to SAW) on a graph $G$ is a walk that visits each vertex at most once. The concept was originally introduced to model polymer molecules (see Flory \cite{Flory53}), and it soon attracted the interest of mathematicians and physicists. Despite the simple definition, SAWs have been difficult to study and many of the most basic questions regarding them remain unresolved. For a comprehensive introduction the reader can consult e.g.\ \cite{LecturesSAW,MaSaSAWBook}.

A \textit{self-avoiding polygon} (abbreviated to SAP) is a walk that starts and ends at the same vertex, visits the starting vertex exactly twice (at the start and at the end) and every other vertex at most once. We identify two SAPs when they share the same set of edges. Fundamental quantities in the study of SAWs and SAPs are their \textit{connective constants}, $$\mu_w:=\limsup_{n\to \infty} (c_n)^{1/n} \quad \text{  and  } \quad \mu_p:=\limsup_{n\to \infty} (p_n)^{1/n},$$
where $c_n$ and $p_n$ denote the number of SAWs and SAPs of length $n$, respectively, starting from the origin $o$. 
We note that for vertex-transitive graphs, a standard subadditivity argument shows that the limit of $(c_n)^{1/n}$ exists and $(\mu_w)^n\leq c_n$ \cite{Poor}.
It is well known that for Euclidean lattices 
$\mu_p=\mu_w$ \cite{Ham61,KestenI}. On the other hand, it is believed that the strict inequality $\mu_p<\mu_w$ holds for a large class of non-Euclidean lattices, namely non-amenable vertex-transitive graphs. In the current paper we prove that the strict inequality holds for the regular tessellations of the hyperbolic plane.

Except for trivial cases, the only graph for which the connective constant $\mu_w$ is known explicitly is the hexagonal lattice \cite{SAWHex}, and a substantial part of the literature on SAWs is devoted to numerical upper and lower bounds for $\mu_w$. See \cite{Alm,Jensen,SAWCubic,PonTitt} for some work in this direction. In this paper we prove new bounds for the connective constants of SAWs and SAPs on the regular tessellations of the hyperbolic plane, improving those of Madras and Wu \cite{MaWuSAW}. 

The natural questions about SAWs concern the asymptotic rate of growth of the number of SAWs of length $n$, and the distance between the start and the end of a typical SAW of length $n$. These questions have been studied extensively on Euclidean lattices, and substantial progress has been made in the case of the hypercubic lattice $\mathbb{Z}^d$ for $d\geq 5$ by the seminal work of Hara and Slade \cite{HaSla1992,HaSlaCri}. The low-dimensional cases are more challenging, and the gap between what is known and what is conjectured is very large. See \cite{PolyIns,SAWEnds,DuHaSub,HammersleyWelsh,KestenI,KestenII} for some of the most important results.

Recently, the study of SAW on non-Euclidean lattices has received increasing attention. In a series of papers \cite{GriLi2013,GriLi2014b,GriLi2014,GriLi2015,GriLi2016,GriLi2017,GriLi2017b,GriLi2017c}, Grimmett and Li initiated a systematic study of SAWs on vertex-transitive graphs. Their work is primarily concerned with properties of the connective constant.
Madras and Wu \cite{MaWuSAW} proved that $\mu_p<\mu_w$ for some tessellations of the hyperbolic plane. Moreover, they proved that the SAW is ballistic, and the number of $n$-step SAWs grows as $(\mu_w)^n$ within a constant factor.
Hutchcroft \cite{Hutch2017c} proved that the SAW on graphs whose automorphism group has a transitive nonunimodular subgroup satisfies the same properties as well. See \cite{BenSAW,GilSe,Extendable,LeLi,Li2016,NaPe} for other works on non-Euclidean lattices.

In the current paper we study SAWs and SAPs on the regular tessellations of the hyperbolic plane, i.e.\ tilings of the hyperbolic plane by regular polygons of the same degree with the property that the same number of polygons meet at each vertex. The regular tessellations of the hyperbolic plane can be characterised by two positive integers $d$ and $k$, where $d$ is the vertex degree, and $k$ is the face degree (the number of edges of the polygon). The $(d,k)$-regular tessellation of the hyperbolic plane is denoted by $\pazocal{H}(d,k)$. It is well known that $(d-2)(k-2)>4$ for any $\pazocal{H}(d,k)$. 
We remark that for every $\pazocal{H}(d,k)$, any subgroup of its automorphism group is unimodular, because it is countable \cite[Proposition 8.9, Lemma 8.43]{LyonsBook}. Hence the results of Hutchcroft do not apply in our case. 

The main objective of this paper is to extend the results of Madras and Wu to all regular tessellations of the hyperbolic plane. Along the way, we also correct an error in their proofs -- see section \ref{correction} for more details. Our first result states that $p_n$ is exponentially smaller than $c_n$ as $n\to \infty$.

\begin{theorem}\label{main theorem}
For every regular tessellation $\pazocal{H}(d,k)$ of the hyperbolic plane we have $\mu_p\big(\pazocal{H}(d,k)\big)<\mu_w\big(\pazocal{H}(d,k)\big)$.
\end{theorem}
\noindent
In the process we derive some new bounds for $\mu_p$ and $\mu_w$, improving those of \cite{MaWuSAW}.

Enhancing a well known idea of Kesten \cite{KestenBook}, we use an auxiliary model of mixed percolation to obtain upper bounds for $p_n$. To this end, we prove certain isoperimetric inequalities for SAPs, comparing their size with the size of their \textit{inner vertex boundary} and the number of their \textit{inner chords}. See Section \ref{sec-up} for the relevant definitions and \Lr{main lemma} for a precise formulation of the isoperimetric inequality. 

Our lower bounds for $\mu_w$ follow from studying a certain class of `almost Markovian' SAWs. We partition the vertices of our graphs into `concentric' cycles, and we consider SAWs which after arriving at a new cycle are allowed to either move within the same cycle or move to the next cycle. Since balls in our graphs grow exponentially fast, we expect that a typical SAW will behave most of the time in such a manner, hence the connective constant of the aforementioned class of SAWs should approximate $\mu_w$ well. As it turns out, this is asymptotically correct as $d\to\infty$, and enables us to prove that $\mu_w(\pazocal{H}(d,k)\big)$ grows like $d-1$ for fixed $k$. On the other hand, bounding $\mu_p$ by the growth rate of the total number of (not necessarily self-avoiding) walks that return to the origin shows that $\mu_p\leq c\sqrt{d}$ for some constant $c>0$. Madras and Wu \cite{MaWuSAW} improved this naive bound by proving that $\mu_p\leq c'\sqrt{d}$ for some constant $c'<c$. As we will see, our upper bounds for $p_n$ imply that $\mu_p$ remains in fact bounded:

\begin{theorem}\label{asymptotics}
For every integer $k\geq 3$, $\mu_w\big(\pazocal{H}(d,k)\big)=d-1-O(1/d)$ as $d\rightarrow \infty$. Moreover, $\frac{k-1}{k-2}\leq \mu_p\big(\pazocal{H}(d,k)\big)\leq 5$ for any tessellation $\pazocal{H}(d,k)$ of the hyperbolic plane.
\end{theorem}

The asymptotic expansion of $\mu_w$ has been studied extensively in the case of the hypercubic lattice $\mathbb{Z}^d$. Kesten \cite{KestenII} proved the asymptotic expansion $\mu_w=2d-1-1/2d+O(1/d^2)$ using finite memory walks. Since then, several terms of the asymptotic expansion have been computed using the lace expansion (see for example \cite{CliLiSla,HaraSlade}). \\ \indent
We define $\mathbb{P}_n$ to be the uniform measure on SAWs of length $n$ in $\pazocal{H}(d,k)$ starting at the origin $o$, and denote $\big(\omega(0),\omega(1),\ldots,\omega(n)\big)$ the random SAW sampled from $\mathbb{P}_n$. 

\begin{theorem}\label{ballistic}
For every regular tessellation $\pazocal{H}(d,k)$ of the hyperbolic plane, there exist $\varepsilon>0$ and $c>0$ such that
$$\mathbb{P}_n \big(d(o,\omega(n))\geq \varepsilon n \big) \geq 1-e^{-cn} $$ 
for every $n\geq 0$. 
Moreover, for every $\pazocal{H}(d,k)\neq \pazocal{H}(3,7),\pazocal{H}(7,3)$, there exists a constant $M\geq 1$ such that
\begin{align}\label{exponent}
(\mu_w)^n\leq c_n\leq M(\mu_w)^n
\end{align}
for every $n\geq 0$.
\end{theorem}

The ballisticity of the SAW will follow from the next result which applies to arbitrary vertex-transitive graphs, and simplifies the task of showing that the SAW is ballistic.

\begin{theorem}\label{transitive}
Let $G$ be a vertex-transitive graph such that $\mu_p<\mu_w$. Then there exist $\varepsilon>0$ and $c>0$ such that
$$\mathbb{P}_n \big(d(o,\omega(n))\geq \varepsilon n \big) \geq 1-e^{-cn} $$ 
for every $n\geq 0$. 
\end{theorem}

Beyond ballisticity, we expect that the inequality $\mu_p<\mu_w$ implies also \eqref{exponent}. For graphs of superlinear volume growth, it seems possible that the inequality $\mu_p<\mu_w$ has further consequences. For example it should imply that the bubble diagram is finite. See \cite[Chapter 4.2]{LecturesSAW} for a definition of the latter.

We remark that the ballisticity of the SAW implies that it has linear expected displacement for every hyperbolic tessellation $\pazocal{H}(d,k)$. This has been proved for $\pazocal{H}(7,3)$ by Benjamini \cite{BenSAW}. The same result has recently been proved for continuous SAWs in hyperbolic spaces \cite{BenPa}.

We end this introduction with the following question.
 
\begin{question}
Consider an integer $k\geq 3$. Does $\mu_p\big(\pazocal{H}(d,k)\big)$ converge as $d\to \infty$? What is the limit?
\end{question}

We believe that $\frac{k-1}{k-2}$, the lower bound for $\mu_p$ appearing in \Tr{asymptotics}, is a likely candidate for the limit. We remark that our method gives $\Big(\frac{k-1}{k-2}\Big)^{\frac{k-1}{k-2}}$ as an upper bound for the limit.

\section{Preliminaries}

\subsection{Walks}

We define formally some notions appearing in the Introduction, and we also fix some notation.

Consider a locally finite graph $G$. Throughout this paper, we 
fix a vertex $o$ of $G$. For $n=0,1,2,\ldots$, a \textit{walk} of length $n$ is a sequence $(\omega(0),
\omega(1),\ldots, \omega(n))$ of vertices of $G$ such that $\omega(i)$ and $\omega(i+1)$ are neighbours for every $i=0,1,
\ldots,n-1$. A \textit{self-avoiding walk} is a walk, all vertices of which are 
distinct. We denote $c_n$ the number of SAWs of length $n$ starting at $o$ and $c_n(x,y)$ the number of SAWs of length $n$ such that $
\omega(0)=x$ and $\omega(n)=y$.

A \textit{self-avoiding polygon} of length $n$ is a walk in which $\omega(0)=\omega(n)$ and all other vertices are distinct. We identify two SAPs whenever they share the same set of edges. Given a SAW $P$, we write $|P|$ for its length. We write $p_n$ for the number of SAPs of length $n$ containing $o$. We remark that this is not the counting used in some other sources, such as \cite{LecturesSAW}. For example, in $\mathbb{Z}^2$, we have $p_4=4$ and $p_6=12$.
 
We define a \textit{non-backtracking walk} of length $n$ as a walk $(\omega(0),\omega(1),\ldots,$ $\omega(n))$, $n=0,1,2,\ldots$, such that $\omega(i)\neq \omega(i+2)$ for every $0\leq i\leq n-2$. In other words, non-backtracking walks are walks that do not traverse back on an edge they just walked on. A \textit{closed non-backtracking walk} is a non-backtracking walk with the same first and last point. We denote $p_{n,2}$ to be the number of closed non-backtracking walks of length $n$ starting at $o$, and we define
$$\mu_{p,2}=\limsup_{n\to \infty} \big(p_{n,2}\big)^{1/n}.$$

\subsection{Percolation}

We recall some standard definitions of percolation theory in order to fix our notation. For more details the reader can consult e.g.\ \cite{Grimmett,LyonsBook}. 

Consider a locally finite graph $G=(V,E)$. Let $\Omega:= \{0,1\}^E$ be the set of \textit{percolation instances} on $G$. We say that an edge $e$ is \textit{closed} (respectively, \textit{open}) in a percolation instance $\omega\in \Omega$, if $\omega(e)=0$ (resp.\ $\omega(e)=1$). In Bernoulli \textit{bond percolation} with parameter $p\in [0,1]$, each edge is open with probability $p$ and closed with probability $1-p$, with these decisions being independent of each other.

To define \textit{site percolation} we repeat the same definitions, except that we now let $\Omega:= \{0,1\}^V$, and work with vertices instead of edges.

\subsection{Interfaces}

Consider a hyperbolic tessellation $\pazocal{H}(d,k)$. Let $H$ be any finite connected induced subgraph containing $o$. The complement of $H$ contains exactly one infinite component denoted $H_{\infty}$. The \textit{\pint} of $H$ is the pair $(M,B)$ where $M$ is the set of vertices in $H$ adjacent to $H_{\infty}$, and $B$ is the set of vertices in $H_{\infty}$ adjacent to $H$. We call $B$ the \textit{outer boundary} of the interface.

The notion of \pint s was introduced in \cite{GeoPaAnalytic} in order to prove that in 2-dimensional Bernoulli pecolation, several percolation observables are analytic functions of the parameter in the supercritical interval $(p_c,1]$. Their properties have been further studied in \cite{ExpGrowth} and \cite{SitePercoPlane}.

\subsection{Cheeger constant and spectral radius}

Consider an infinite, locally finite, connected graph $G$. The \textit{adjacency matrix} of $G$ is the matrix $A$ such that its $(x,y)$ entry $A(x,y)$ is one when $x$ and $y$ are connected with an edge, and zero otherwise. The quantity 
$$R=R(G)=\limsup_{n\to \infty} \big(A^n(x,y)\big)^{1/n}$$
does not depend on the choice of $x$ and $y$, and is called the \textit{spectral radius} associated to $A$. 

Let $K$ be a set of vertices of $G$. The \textit{edge boundary} $\partial^E K$ of $K$ is the set of edges of $G$ with exactly one endvertex in $K$. The \textit{edge Cheeger constant} of $G$ is defined by 
$$h=h(G)=\inf \left \{\dfrac{|\partial^E K|}{|K|}\right \}.$$
where the infimum ranges over all finite sets of vertices $K$. The edge Cheeger constant of $\pazocal{H}(d,k)$ has been calculated by H{\"a}ggstr{\"o}m, Jonasson and Lyons in \cite{HagJoLy} and is given by
\begin{equation}\label{edge ch}
h=(d-2)\sqrt{1-\dfrac{4}{(d-2)(k-2)}}.
\end{equation}
We remark that the vertex Cheeger constant of the hyperbolic tessellations $\pazocal{H}(d,k)$ with $k=3$ or $4$ has been computed in \cite{SitePercoPlane}.
The edge Cheeger constant and the spectral radius on $G$ are related via the inequality
\begin{equation}\label{spectral}
h^2+R^2\leq d^2,
\end{equation}
where now $d$ denotes the maximum degree of $G$. The above inequality has been proved by Mohar in \cite{MoharSpectral}.

\section{The results of Madras and Wu}\label{correction}

In this section, we will focus on the results of Madras and Wu mentioned in the Introduction. Let us start by mentioning the following result regarding $\mu_{p,2}$. 

\begin{proposition}[{\cite[Theorem 6.10]{LyonsBook}}]\label{non-reversing}
Let $R$ be the spectral radius of a $d$-regular connected graph. Assume that $R>2\sqrt{d-1}$. Then
$$\mu_{p,2}=\dfrac{R+\sqrt{R^2-4(d-1)}}{2}.$$
\end{proposition}

Madras and Wu proved the following results in \cite{MaWuSAW}.

\begin{lemma}[\cite{MaWuSAW}]\label{uniform bound}
Consider a hyperbolic tessellation $\pazocal{H}(d,k)$. Then, for every $n\geq 0$ and every pair of vertices $x$ and $y$
$$c_n(x,y)\leq \dfrac{d^2}{d-1} \mu_{p,2}^{n+k-1}.$$
\end{lemma}

\begin{theorem}[\cite{MaWuSAW}]\label{hypothesis}
Consider a hyperbolic tessellation $\pazocal{H}(d,k)$. Assume that there exist constants $M$ and $\rho$ such that $\rho<\mu_w$ and 
\begin{align}\label{uniform decay}
c_n(x,y)\leq M \rho^n.
\end{align}
Then there exists a constant $M'$ such that 
$$(\mu_w)^n \leq c_n \leq M'(\mu_w)^n$$
for every $n\geq 0$.
\end{theorem}

The proofs of these results use specific properties of hyperbolic tessellations, so an extension of these results to arbitrary vertex-transitive graphs requires new ideas.   

Madras and Wu claim to have verified \eqref{uniform decay} for all hyperbolic tessellations $\pazocal{H}(d,k)$ with $(d,k)$ satisfying one of the following conditions:
\begin{enumerate}
\item $k=3$, $d\geq 10$,
\item $k=4$, $d\geq 6$,
\item $k=5$, $d\geq 5$,
\item $k\in \{6,7,8,9\}$, $d\geq 4$,
\item $k\geq 10$, $d\geq 3$.
\end{enumerate}
A key result to this end is Proposition $2.1$ in \cite{MaWuSAW}, which states that for every $k>4$, $\mu_w(\pazocal{H}(d,k))\geq \big((d-1)^{k-4} (d-2)\big)^{1/(k-3)}$. However, the proof of the later statement does not apply to hyperbolic tessellations of degree $3$. 

To make this more apparent, let us describe their argument. Consider some $\pazocal{H}(d,k)$ with $k>4$. First partition the vertices of the graph into layers as follows. Fix some face, and let the first layer consist of the vertices of this face. The second layer  consists of those vertices which are not in the first layer but on a face which has a vertex in common with the first layer. The third and every subsequent layer are formed in a similar way. Given a vertex $x$, we let $x+i$ and $x-i$ denote the $i$th vertex along the same layer on $x$ on the clockwise, anticlockwise direction, respectively. Starting at a vertex in the first layer, define a SAW according to the following rules. At the first step, the walk is allowed to move to a next layer neighbour, and each time the walk reaches a vertex on a new layer, it is allowed to move either to the next layer or within the same layer. Whenever it reaches a vertex within the same layer, it is allowed to move only to the next layer. It is claimed in \cite{MaWuSAW} that each time the walk reaches a vertex $x$ on a new layer, then both $x+i$ and $x-i$ have no neighbour in the previous layer and $d-2$ neighbours in the next layer for $i=1,2,\ldots,k-4$. 

If $d>3$, then this is true, but not always when $d=3$, since it is possible that either $x+k-4$ or $x-(k-4)$ have a neighbour in the previous layer. Based on the above erroneous observation, the rules of the SAWs are modified by allowing them to move within the layers until they visit $x+k-4$ or $x-(k-4)$, and from there the walks are allowed to move to the next layer. These walks give the lower bound $\big((d-1)^{k-4} (d-2)\big)^{1/(k-3)}$ mentioned above. This proof works whenever $d>3$, but not when $d=3$.

The aforementioned error can easily be corrected, still yielding a useful lower bound for $\mu_w$ when $d=3$. We claim that each time the walk reaches a vertex $x$ on a new layer, then both $x+i$ and $x-i$ have no neighbour in the previous layer for every $i=1,2,\ldots,k-5$. To see this, let us write $\pazocal{H}(3,k)^*$ for the dual graph of $\pazocal{H}(3,k)$ and consider the faces of $\pazocal{H}(3,k)$ between consecutive layers $n$ and $n+1$. We define the $n$th layer of $\pazocal{H}(3,k)^*$ to be the corresponding dual vertices. Then
each vertex of $\pazocal{H}(3,k)^*$ has at least $k-4$ next layer neighbours. See \Lr{neigh-3} and \Lr{span cycle}. The claim follows now easily.

We now have that $c_{(k-4)n}\geq 2^{(k-5)n}$ which implies the following bound.
\begin{proposition}\label{prop corrected}
For every $k\geq 7$, $\mu_w(\pazocal{H}(3,k))\geq 2^{(k-5)/(k-4)}$.
\end{proposition}
Combining this bound with Proposition \ref{non-reversing} and \Lr{uniform bound} we obtain that the hypothesis and the conclusion of \Tr{hypothesis} hold for every $\pazocal{H}(3,k)$ with $k\geq 11$.

We can now conclude that the results of Madras and Wu hold for those hyperbolic tessellations $\pazocal{H}(d,k)$ with $(d,k)$ satisfying and one of the following conditions:
\begin{enumerate}
\item $k=3$, $d\geq 10$,
\item $k=4$, $d\geq 6$,
\item $k=5$, $d\geq 5$,
\item $k\in \{6,7,8,9,10\}$, $d\geq 4$,
\item $k\geq 11$, $d\geq 3$.
\end{enumerate}
The set of tessellations $\pazocal{H}(d,k)$ with $(d,k)$ as above is denoted $\pazocal{L}$. Comparing $\pazocal{L}$ with the original list of tessellations of Madras and Wu mentioned above, we see that the latter contains one more tessellation, namely $\pazocal{H}(3,10)$.

We now gather all the valid bounds for $\mu_w$ of Madras and Wu.

\begin{proposition}[\cite{MaWuSAW}]\label{lower bounds}
For SAWs on $\pazocal{H}(d,k)$, we have
\begin{enumerate}
\item for $k=3$, $\mu_w\geq \sqrt{(d-2)(d-3)}$,
\item for $k=4$, $\mu_w\geq \big((d-1)(d-2)^2\big)^{1/3}$,
\item for $k\geq 5$ and $d>3$, $\mu_w\geq \big((d-1)^{k-4} (d-2)\big)^{1/(k-3)}$.
\end{enumerate}
\end{proposition}

We remark that a careful implementation of the arguments of the proof of \Tr{main theorem} yields even better lower bounds for $\mu_w$.

\section{$\mu_p<\mu_w$ implies ballisticity}

We will start by proving \Tr{transitive}, which will be used in the proof of \Tr{ballistic}. This result is also of independent interest as it applies to all vertex-transitive graphs. The rough idea of the proof is the following: if the endpoint of a SAW is at a sublinear distance from
the origin, then we can close it into a closed walk (not necessarily self-avoiding) by forcing a geodesic. The intersection
points cut the walk in several SAPs (and some SAWs lying in the geodesic that are traversed in both directions) but only sublinearly many, and then we use that the number of SAPs is exponentially smaller than the number of SAWs.

\begin{proof}[Proof of \Tr{transitive}]
Let $n\in \mathbb{N}$, and let $1>\varepsilon>0$. Consider a vertex $y\neq o$ such that $r:=d(o,y)< \varepsilon n$, and fix a \textit{geodesic} $X=(X(0)=o,X(1),\ldots,X(r)=y)$ from $o$ to $y$, i.e.\ a SAW from $o$ to $y$ that has length $d(o,y)$. Let $W=(\omega(0)=o,\omega(1)\,\ldots,\omega(n)=y)$ be a SAW from $o$ to $y$. We emphasize that the vertices of $X$ and $W$ are ordered. Given a vertex $v=X(k)$ of $X$, we let $U(v)$ be the set of vertices $X(l)$ with $l>k$ lying in both $X$ and $W$. 

Traversing $W$ from $o$ to $y$, and then traversing $X$ from $y$ to $o$, we obtain a closed walk starting and ending at $o$. We will decompose the edge set of this closed walk into some `almost' edge-disjoint SAWs $W_1,W_2,\ldots, W_N$ and SAPs $P_1,P_2,\ldots,P_M$ as follows. Set $x_0=o$, and for $j\geq 1$ define inductively $W_j$ to be the maximal subwalk of $X$ starting from $x_{2j-2}$ towards $y$ such that $W_j$ is also a subwalk of $W$, and $x_{2j-1}$ to be the last vertex of $W_j$. We remark that $W_j$ can have length $0$, in which case $x_{2j-1}$ coincides with $x_{2j-2}$. We also define $x_{2j}$ to be the vertex of $U(x_{2j-1})$ nearest to $x_{2j-1}$, and $P_j$ to be the concatenation of the subwalk of $W$ from $x_{2j-1}$ to $x_{2j}$ and the subwalk of $X$ from $x_{2j}$ to $x_{2j-1}$. The latter subwalk is denoted $S_j$. Observe that by definition of $U(x_{2j-1})$, $x_{2j}$ is not contained in the subwalk of $X$ from $x_0$ to $x_{2j-1}$. We stop once some $W_j$ or $P_j$ contains $y$. Clearly either $N=M$ or $N=M+1$. See Figure \ref{deco}. We will write $S_w(W)$ for $\{W_1,W_2,\ldots,W_N\}$ and $S_p(W)$ for $\{P_1,P_2,\ldots,P_M\}$.

We will first show that there is only one SAW from $o$ to $y$ which gives rise  to the pair $(S_w(W),S_p(W))$ in the above procedure, namely $W$. Indeed, notice that we can identify $W_j$ as the $j$th element of $S_w(W)$ closest to $o$, and similarly we can identify $P_j$ as the $j$th element of $S_p(W)$ closest to $o$. Now we can identify $S_j$ as the subwalk of $X$ from the start of $W_{j+1}$ to the end of $W_j$. Then $W$ is simply the union of $W_1,W_2,\ldots, W_N$ with $P_1\setminus S_1,P_2\setminus S_2,\ldots, P_M\setminus S_M$.

\begin{figure}
\centering

\begin{tabular}{@{}c@{}}
   \includegraphics[width=.5\linewidth]{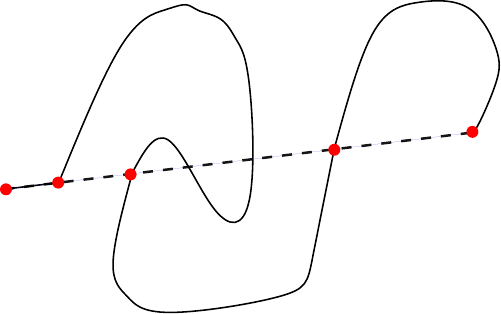}
   \put(-18,63){$x_6$}
   \put(-66,57){$x_4$}
   \put(-152,46){$x_2$}
   \put(-190,42){$x_1$}
   \put(-210,40){$x_0$}
   \vspace{.2cm}
\end{tabular}

\vspace{\floatsep}

\begin{tabular}{@{}c@{}}
   \includegraphics[width=.6\linewidth]{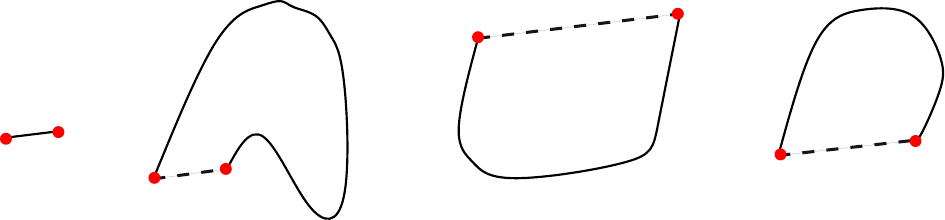}
   \put(-9,13){$x_6$}
   \put(-51,9){$x_4$}
   \put(-80,59){$x_4$}
   \put(-126,53){$x_1$}
   \put(-218,3){$x_1$}
   \put(-192,4){$x_2$}
   \put(-235,14){$x_1$}
   \put(-252,12){$x_0$}
   \put(-247,29){$W_1$}
   \put(-205,-2){$S_1$}
   \put(-105,59){$S_2$}
   \put(-31,6){$S_3$}
\end{tabular}
\caption{An illustration of the decomposition. The edges of the SAW appear in solid lines and the edges of the geodesic appear in dashed lines. Here $W_1$ is depicted as the solid straight line from $x_0$ to $x_1$, while $W_2$ and $W_3$ are the single vertices $x_2$ and $x_4$, respectively, because $x_3$ coincides with $x_2$ and $x_5$ coincides with $x_4$. Moreover, $S_1$, $S_2$ and $S_3$ are the dashed straight lines connecting $x_1$ to $x_2$, $x_2$ to $x_4$ and $x_4$ to $x_6$, respectively.}
\label{deco}
\end{figure}

Let us count the possibilities for the set $\{W_1,W_2,\ldots,$ $W_N\}$ of SAWs produced in this way. Each $W_i$ is a possibly edge-less subwalk of $X$ and it is determined by its start and end. Any set of vertices of $X$ gives rise to a collection of edge-less SAWs, while any set of vertices $\{y_1,y_2,\ldots,y_{2k}\}$ of $X$ of even cardinality with $d(y_i,o)<d(y_{i+1},o)$ gives rise to a collection of SAWs with at least one edge by considering the subwalks of $X$ that start at $y_{2j-1}$ and end at $y_{2j}$. Thus there are at most $4^{\varepsilon n}$ possibilities for $\{W_1,W_2,\ldots,$ $W_N\}$. 

It remains to find an upper bound for the number of sets of SAPs $\{P_1,P_2,\ldots,P_M\}$ produced in this way. Let $r\in (\mu_p,\mu_w)$. Then there is a constant $C>1$ such that
$p_n\leq Cr^n$ for every $n\geq 1$. We claim that $M\leq \varepsilon n$ and $\sum_{i=1}^M |P_i|\leq (1+\varepsilon)n$, from which it easily follows that there are at most $$\sum_{(n_1,n_2,\ldots,n_l)}C^l r^m$$ possibilities for $\{P_1,P_2,\ldots,P_M\}$, where the sum ranges over all compositions of positive integers $m\leq (1+\varepsilon)n$ into at most $\varepsilon n$ positive parts. For the first part of the claim, notice that the subwalks $S_i$ are edge-disjoint, hence by associating to each $P_i$ an edge of $S_i$, we conclude that $M\leq \varepsilon n$, as desired. For the second part of the claim, notice that each edge of $\big(\bigcup_{i=1}^M P_i\big)\setminus X$ lies in one of $P_1,P_2,\ldots,P_M$, so each such edge contributes $1$ to $\sum_{i=1}^M |P_i|$. Since the subwalks $S_i$ are edge-disjoint, it follows that any edge of $X$ lies in at most one $S_i$, and at most one subgraph of the form $P_i \setminus S_i$. Consequently, each edge of $X$ lies in at most $2$ SAPs of $\{P_1,P_2,\ldots,P_M\}$, hence each such edge contributes at most $2$ to $\sum_{i=1}^M |P_i|$. This easily implies the claim.

Combining the above, we get that
$$c_n(o,y)\leq 4^{\varepsilon n}\sum_{(n_1,n_2,\ldots,n_l)}C^l r^m.$$ 
For every $m$ and $l$, there are 
$${m-1 \choose l-1}\leq \dfrac{(m-1)^{l-1}}{(l-1)!}\leq \Big(\dfrac{m-1}{l-1}\Big)^{l-1} e^{l-1}$$ compositions of $m$ with $l$ elements, where in the second inequality we used that $e^{l-1}\geq \frac{(l-1)^{l-1}}{(l-1)!}$, which follows from the Taylor expansion of $e^{l-1}$. Notice that 
$$\Big(\dfrac{m-1}{l-1}\Big)^{l-1} e^{l-1}\leq \Big(\dfrac{(1+\varepsilon)en}{l-1}\Big)^{l-1} \leq \Big(\dfrac{(1+\varepsilon)e}{\varepsilon}\Big)^{\varepsilon n},$$
whenever $l\leq \varepsilon n$ and $m\leq (1+\varepsilon)n$, where in the second inequality we used that for every $y>0$, the function $f(x):=\Big(\frac{y}{x}\Big)^x$ is increasing on $(0,y/e]$, and that $l\leq \varepsilon n< (1+\varepsilon)n$.
Hence the number of compositions of all positive integers $m\leq (1+\varepsilon)n$ into at most $\varepsilon n$ positive parts is bounded from above by 
$$\varepsilon (1+\varepsilon)n^2 \Big(\dfrac{(1+\varepsilon)e}{\varepsilon}\Big)^{\varepsilon n}.$$ We can now deduce that 
$$c_n(o,y)\leq \varepsilon(1+\varepsilon)n^2 \Big(\dfrac{(1+\varepsilon)e}{
\varepsilon}\Big)^{\varepsilon n} 4^{\varepsilon n} C^{\varepsilon n} r^{(1+\varepsilon)n}.$$

It is not hard to see that the number of vertices in the ball of radius $\epsilon n$ of $o$ is at most $d^{\epsilon n}$, where $d$ is the degree of the graph. Therefore,
$$\sum_{y:d(o,y)<\epsilon n} c_n(o,y)\leq \varepsilon(1+\varepsilon)n^2 \Big(\dfrac{(1+\varepsilon)e}{\varepsilon}\Big)^{\varepsilon n} d^{\epsilon n} 4^{\varepsilon n} C^{\varepsilon n} r^{(1+\varepsilon)n}.$$
Since $r<\mu_w$, we can choose $\varepsilon>0$ small enough so that the desired assertion holds.
\end{proof}

\begin{remark}
Notice that we barely used the transitivity of $G$ in the proof of \Tr{transitive}. It is not hard to see that the proof works for all bounded degree graphs $G$ such that  
$$\limsup_{n\to \infty} \Big(\sup_{x\in V(G)} c_n(x,x)\Big)^{1/n}<\liminf_{n\to \infty}\Big(\inf_{x\in V(G)} c_n(x)\Big)^{1/n},$$
where $c_n(x)$ denotes the number of SAWs of length $n$ starting from $x$.
\end{remark}

\section{Upper bounds for $\mu_p$}\label{sec-up}

The aim of this section is to obtain upper bounds for $\mu_p$. As it will become apparent later on, there is a certain result for which we need to work with SAPs living on graphs other than $\pazocal{H}(d,k)$. 

To define the family of graphs with which we are going to work, we first need to introduce some notions. Given a finite plane graph $G$, the \textit{boundary} of $G$ is the set of vertices and edges incident with the unbounded face of $G$. An \textit{internal} vertex of $G$ is a vertex not in the boundary of $G$. We say that a face of $G$ has a \textit{simple boundary} if the vertices and edges incident with this face define a SAP.

Now given some positive integers $d$ and $k$ with $(d-2)(k-2)>4$, we define $\pazocal{S}(d,k)$ to be the set of finite connected plane graphs $G$ with the following properties:
\begin{enumerate}
\item the boundary of $G$ is a SAP,
\item all bounded faces of $G$ have a simple boundary and face degree $k$,
\item and every internal vertex of $G$ has degree $d$.
\end{enumerate}

Clearly, any SAP of $\pazocal{H}(d,k)$ defines a graph in $\pazocal{S}(d,k)$ by removing all vertices and edges not in the region bounded by $P$. However, not every graph in $\pazocal{S}(d,k)$ can be obtained in this way.

In order to obtain the desired upper bounds for $\mu_p$, we will need to first prove an isoperimetric inequality for SAPs of $\pazocal{S}(d,k)$.
To this end, let us introduce the following concepts.

\vspace{.1cm}

Given a SAP $P$ of a graph $G\in\pazocal{S}(d,k)$, the \textit{interior} of $P$ is the set of edges and vertices of $G$ not visited by $P$ which lie in the region bounded by $P$. In this definition, we treat edges as open segments, so that edges can lie in the interior of $P$ even if they have an endvertex in $P$. The \textit{inner chords} $\operatorname{ch}(P)$ of $P$ are those edges in the interior of $P$ both endvertices of which are in $P$. The \textit{inner vertex boundary} $\partial^V P$ of $P$ is the set of vertices in the interior of $P$ lying in a face that is incident to $P$. We write $I$ for the set of vertices lying in the interior of $P$. See Figure \ref{chords-vertex boundary}. We let $|P|$ denote the length of $P$, $|\operatorname{ch}(P)|$ denote the number of edges of $\operatorname{ch}(P)$, and $|\partial^V P|$ denote the number of vertices of $\partial^V P$. In general, $|\cdot|$ denotes the cardinality of a set.

\begin{figure}
\centering
\includegraphics{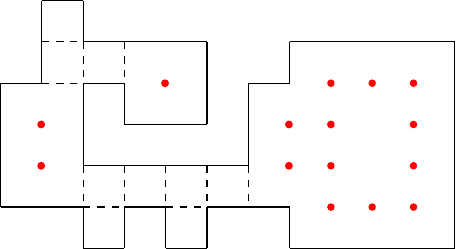}
\caption{A SAP in the square lattice $\mathbb{Z}^2$ together with its inner chords and its inner vertex boundary. The SAP is depicted in solid lines, the inner chords are depicted in dashed lines, and the vertices of the inner vertex boundary are depicted as single points.}
\label{chords-vertex boundary}
\end{figure}

\begin{theorem}\label{main lemma}
Let $G$ be a graph in $\pazocal{S}(d,k)$ and let $P$ be the SAP at its boundary. Then $$(k-2)|\operatorname{ch}(P)|+\big( (d-2)(k-2)-3 \big)|\partial^V P|\leq |P|-k.$$
\end{theorem}

This isoperimetric inequality will be first proved in the cases $\partial^V P=\emptyset$ and $\operatorname{ch}(P)=\emptyset$, where the isoperimetric inequality simplifies to $$(k-2)|\operatorname{ch}(P)|\leq |P|-k$$ and $$\big( (d-2)(k-2)-3 \big)|\partial^V P|\leq |P|-k.$$ Then we will decompose any SAP into SAPs with empty inner vertex boundary and to SAPs with no inner chords. We will first need a series of lemmas.

The first lemma will be used to handle the case $\partial^V P=\emptyset$. It is inspired by Lemma 2.1 in \cite{AnBeHor}. 

\begin{lemma}\label{edges enum}
Let $G$ be a graph in $\pazocal{S}(d,k)$ and let $P$ be the SAP at its boundary. Write $m$ for the number of directed edges $(x,y)$ in the interior of $P$ with $x$ on $P$. Then
$$m=\dfrac{2|P|-2k-\big( (d-2)(k-2)-4 \big)|I|}{k-2}.$$
\end{lemma}
\begin{proof}
Let $V$, $E$ and $F$ be the number of vertices, edges and faces, respectively, of $G$. Using Euler's formula we obtain $V-E+F=1$ (because we are not counting the unbounded face of $G$). Since every edge of $G$ except from those of $P$ are incident to two faces, we get $kF=2E-|P|$, and hence $$E=\dfrac{kV-|P|-k}{k-2}.$$ Summing vertex degrees gives $2E=d|I|+2|P|+m$. Clearly $V=|P|+|I|$, and the assertion follows.
\end{proof}

The second lemma will be used to handle SAPs with no inner chords and \textit{connected} inner vertex boundary, while the third lemma will be used to divide SAPs with no inner chords into SAPs with no inner chords and connected inner vertex boundary. We will first need some definitions.

Given a SAP $P$, let $\Gamma$ be the graph spanned by $\partial^V P$. The \textit{components} of $\partial^V P$ are the vertex sets of the connected components of $\Gamma$. We say that $\partial^V P$ is \textit{connected} if $\Gamma$ is connected. 

\begin{figure}
\centering
\includegraphics[width=.4\linewidth]{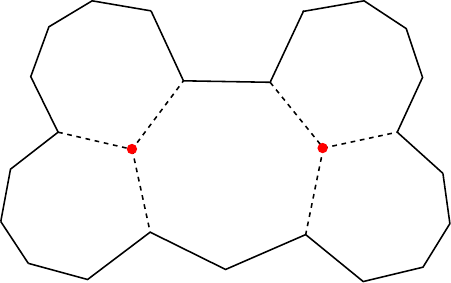}
\caption{An example of a SAP with no inner chords and disconnected inner vertex boundary. The SAP is depicted in solid lines and its inner vertex boundary as red points.}
\label{vpdisconnected}
\end{figure}

Let also $\Gamma^U$ be the graph obtained from $\Gamma$ by keeping all vertices and edges lying in the unbounded face of $\Gamma$, i.e.\ the unbounded region of $\overline{\mathbb{R}^2\setminus E(\Gamma)}$ (not to be confused with the face inherited from $G$). Notice that $\Gamma$ and $\Gamma^U$ have the same vertex sets but it is possible that not all edges of $\Gamma$ are contained in $\Gamma^U$. The latter can happen if $\Gamma^U$ contains a SAP $Q$ and an inner chord of $Q$ lies in $\Gamma$. Given a component $\pazocal{C}$ of $\partial^V P$, we let $\Gamma[\pazocal{C}]^U$ be the subgraph of $\Gamma^U$ spanned by the vertices of $\pazocal{C}$. The \textit{total boundary length} of $\pazocal{C}$ counts all edges in $\Gamma[\pazocal{C}]^U$ exactly once, except for those not incident with a bounded face of $\Gamma[\pazocal{C}]^U$, which are counted twice. In other words, the total boundary length of $\pazocal{C}$ can be computed by walking around $\Gamma[\pazocal{C}]^U$ and counting each edge of $\Gamma[\pazocal{C}]^U$ with multiplicity, namely as many times as it appears. The only edges counted more than once are the \textit{bridges} of $\Gamma[\pazocal{C}]^U$, i.e.\ those edges of $\Gamma[\pazocal{C}]^U$, the deletion of which disconnects $\Gamma[\pazocal{C}]^U$. The total boundary length of $\partial^V P$ is the sum of the total boundary lengths of its components.

The faces of $G$ in the region bounded by $P$ will be called \textit{internal} and every other face is called \textit{external}. The following lemma is inspired by Lemma 2.2 in \cite{AnBeHor}.

\begin{lemma}\label{upper bound}
Let $G$ be a graph in $\pazocal{S}(d,k)$ and let $P$ be the SAP at its boundary. Assume that $\operatorname{ch}(P)=\emptyset$ and that $\partial^V P$ is non-empty and connected. Let also $m$ be as in \Lr{edges enum}, and $n'$ be the total boundary length of $\partial^V P$. Then 
$m=\dfrac{|P|+n'}{k-2}$, and hence $$n'=|P|-2k-\big( (d-2)(k-2)-4 \big)|I|.$$
\end{lemma}
\begin{proof}
Let $S$ be the set of internal faces incident to $P$.
We will show that $(k-2)|S|=|P|+n'$ and $m=|S|$. We first claim that any $f\in S$ has $2$ edges contributing to $m$ and $k-2$ edges contributing to $|P|+n'$. We will show that the subgraph of $\Gamma$ living in $f$ is a path. Once we prove this, the claim will follow immediately, as we are assuming that $\operatorname{ch}(P)=\emptyset$. 

Let us denote this subgraph $W$ and assume to the contrary that $W$ is not a path. Then $W$ is disconnected. Pick two vertices $u$ and $v$ in distinct connected components of $W$. Removing $u$, $v$ and the edges adjacent to them from the boundary of $f$, we obtain two disjoint paths. Each of these two paths contains a vertex of $P$ because otherwise $u$ and $v$ belong to the same component of $W$. Let $x$ and $y$ be vertices of $P$ from each of these two paths. Since $\partial^V P$ is connected, there is a path in $G$ connecting $u$ to $v$ which visits only vertices of $\partial^V P$. We can extend this path to a simple closed curve $\gamma$ in the plane by adding a curve connecting $u$ with $v$ which except at its endpoints, lies in the interior of $f$. Notice that $x$ and $y$ must lie in different components of $\mathbb{R}^2\setminus \gamma$. This means that either $x$ or $y$ belongs to the interior of $P$. This contradiction shows that $W$ is a path.

Recall the definition of $\Gamma^U$. Notice that each edge of $P$ is incident to exactly one face in $S$, each bridge of $\Gamma^U$ is incident to exactly two faces in $S$ and each of the remaining edges of $\Gamma^U$ is incident to exactly one face in $S$. Therefore, $(k-2)|S|=|P|+n'$. Moreover, for each of the $m$ directed edges $(x,y)$ in the interior of $P$ with $x$ on $P$, both faces incident to $(x,y)$ belong to $S$. Hence $m=|S|$ which implies that $m=\dfrac{|P|+n'}{k-2}$, as desired. Applying \Lr{edges enum} we obtain $n'=|P|-2k-\big( (d-2)(k-2)-4 \big)|I|.$
\end{proof}

\begin{figure}
\centering
\hspace{3cm}
\includegraphics[width=0.5\textwidth]{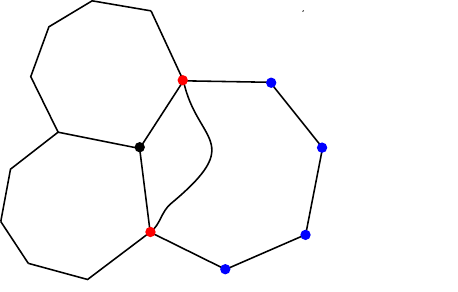}
\caption{The situation in the proof of \Lr{upper bound}. The blue vertices belong to $P$ and the red vertices belong to $W$. The black vertex of $f$ belongs to $P$ or $W$. If it belongs to $P$, then it lies in the interior of $P$, which is absurd.}
\label{path-connected}
\end{figure}

We remark that $\partial^V P$ is always connected when it is non-empty, $\operatorname{ch}(P)=\emptyset$ and $k=3$, but this will not be used in the proofs. On the other hand, when $k>3$, it is possible for $\partial^V P$ to be disconnected.

In the following lemma, given a component $\pazocal{C}$ of $\partial^V P$, we write $\partial \pazocal{C}$ for the set of vertices of $\big(\partial^V P \setminus \pazocal{C} \big) \cup P$ that lie in some face sharing a common vertex with $\pazocal{C}$. Along the way, we will define certain auxiliary graphs and we will refer to faces of those graphs. To avoid any confusion, we will specify each time the graph that the faces belong to.

\begin{lemma}\label{useful lemma}
Let $G$ be a graph in $\pazocal{S}(d,k)$ and let $P$ be the SAP at its boundary. Assume that $\operatorname{ch}(P)=\emptyset$ and $\partial^V P\neq \emptyset$. Let $\pazocal{C}$ be a component of $\partial^V P$. Then the graph spanned by $\partial\pazocal{C}$ is a SAP $Q$ which has no inner chords, and $\partial^V Q=\pazocal{C}$.
\end{lemma}
\begin{proof}
Let $\pazocal{H}$ be the subgraph of $G$ defined by the vertices and edges lying in at least one face of $G$ incident with $\pazocal{C}$. It is not hard to see that $\pazocal{H}$ contains a SAP $Q$ such that the interior of $Q$ contains $\pazocal{C}$ and each vertex of $Q$ lies in a face of $\pazocal{H}$ which is incident to $\pazocal{C}$. Indeed, this is well-known for planar triangulations (see e.g.\ \cite[Proposition 2.1]{KestenBook}) and it follows in our case by triangulating the faces of $\pazocal{H}$ as follows. For every face $f$ of $\pazocal{H}$ incident with $\pazocal{C}$, we pick a vertex $x\in f\cap \pazocal{C}$ and draw an edge between $x$ and every vertex lying in $f$ which is not already adjacent to $x$. Triangulate any other bounded face $f'$ of $\pazocal{H}$ by adding a vertex $u$ in its interior and connecting it with an edge to every vertex of $\pazocal{H}$ lying in $f'$. Then we end up with a planar triangulation $T$ which contains a SAP $Q$ such that the interior of $Q$ contains $\pazocal{C}$ and each vertex of $Q$ is adjacent to $\pazocal{C}$ (here the adjacency takes place in $T$). The latter implies that each vertex of $Q$ lies in a face of $\pazocal{H}$ which is incident to $\pazocal{C}$. It is easy to see that all vertices and edges of $Q$ lie in $\pazocal{H}$ from the way $T$ is defined. 

We claim that 
\begin{equation}\label{inter claim}
Q \text{ has no inner chords}, V(Q)=\partial \pazocal{C} \text{ and }  \pazocal{C}=\partial^V Q,
\end{equation}
which implies the desired result.
For the first assertion of the claim, assume for a contradiction that $Q$ has an inner chord $e$. Then $e$ divides the region bounded by $Q$ into two and by the connectivity of $\pazocal{C}$, one of the two regions contains $\pazocal{C}$ in its interior. Consider some edge $e'\in Q$ that lies in the region that does not contain $\pazocal{C}$ and notice that $e'$ does not lie in a face of $G$ incident with $\pazocal{C}$, which is absurd, since every edge of $\pazocal{H}$ lies in a face of $G$ incident with $\pazocal{C}$. Thus $Q$ has no inner chords.

Let us now show that $V(Q)\subset \partial \pazocal{C}$. Consider a vertex $u\in Q$ and let $f$ be a face of $\pazocal{H}$ incident with both $u$ and $\pazocal{C}$. Such a face exists because of the way $\pazocal{H}$ is defined. It suffices to show that $f$ contains a vertex of $P$ because every vertex lying in $f$ but not in $P$ would then lie by definition in $\partial^V P$. Notice that if $f$ does not contain a vertex in $P$, then every vertex in $f$ is connected to $\pazocal{C}$ by path lying in the interior of $P$, hence every vertex in $f$ lies either in $\pazocal{C}$ or in a bounded face of $\Gamma[\pazocal{C}]^U$. In particular, this is the case for $u$. However, this is absurd because $\pazocal{C}$ lies in the interior of $Q$, hence so does every vertex in a bounded face of $\Gamma[\pazocal{C}]^U$. Therefore, $f$ is incident with $P$, which proves that $V(Q)\subset \partial \pazocal{C}$.

For the reverse direction, notice that each vertex of $\partial \pazocal{C}$ lies either in $Q$ or in the interior of $Q$, hence
we need to show that no vertex of $\partial \pazocal{C}$ lies in the interior of $Q$. Indeed, let us first show that no vertex of $P$ lies in the interior of $Q$. 
Assume to the contrary that some vertex $u\in P$ does.
Notice that any curve $\gamma:[0,\infty)\rightarrow \mathbb{R}^2$ that connects $u$ to infinity, i.e.\ $\gamma(0)=u$ and $\lim_{t\to\infty} |\gamma(t)|=\infty$, must intersect $Q$ because $u$ lies in the interior of $Q$, and that the vertices and edges of $Q$ belong either to $P$ or to the interior of $P$. However, $u$ is incident with the unbounded face of $P$, hence there is a planar curve connecting $u$ to infinity without intersecting $Q$. This contradiction shows no vertex of $P$ lies in the interior of $Q$.

It follows that every vertex in the interior of $Q$ belongs also to the interior of $P$. In order to show that no vertex of $\partial^V P \setminus \pazocal{C}$ lies in the interior of $Q$, it suffices to show that the interior of $Q$ is connected. Let us write $\mathcal{C}$ for the connected component of $\pazocal{C}$ in the interior of $Q$. Assume that a face $f$ of $G$ which is incident with a vertex $x\in \mathcal{C}$ and a vertex $y$ that lies in the interior of $Q$ but not in $\mathcal{C}$.  
Removing $x$ and $y$ from the boundary of $f$, we obtain two paths and both of them need to contain a vertex of $Q$ because $x$ and $y$ are not connected in the interior of $Q$. Let $z,w$ be vertices of $Q$ from each of the two paths. We can find a planar curve $\gamma$ that connects $z$ to $w$ and, expect at its endpoint, lies in the interior of $f$. Then $\gamma$ divides the region bounded by $Q$ into two and one of them contains $x$ in its interior and the other contains $y$ in its interior. By the connectivity of $\pazocal{C}$, one of these regions contains $\pazocal{C}$ entirely. Consider a vertex $v\in Q$ which lies in the region that does not contain $\pazocal{C}$ but not in $f$, and notice that $v$ does not lie in a face of $G$ incident with $\mathcal{C}$, which is absurd. We can now conclude that the interior of $Q$ is connected, which implies that no vertex of $\partial^V P \setminus \pazocal{C}$ lies in the interior of $Q$.

\begin{figure}
\centering
\includegraphics{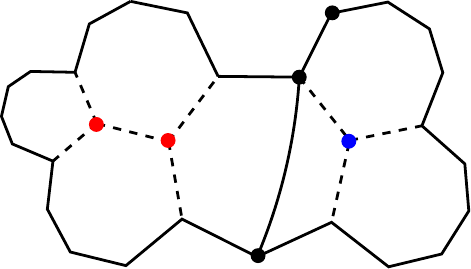}
\put(-155,52){$x$}
   \put(-55,52){$y$}   
   \put(-109,-3){$z$}
   \put(-95,96){$w$}
   \put(-64,128){$v$}
   \put(-98,50){$\gamma$}
\caption{The situation in the proof of \Lr{useful lemma}.}
\label{discimp}
\end{figure}

To prove \eqref{inter claim}, it remains to show that $\pazocal{C}$ coincides with $\partial^V Q$. Indeed, each vertex of $\pazocal{C}$ lies in a face incident with $P$, and each such vertex of $P$ lies in $\partial \pazocal{C}=V(Q)$. Thus $\pazocal{C}\subset \partial^V Q$. For the reverse direction, let $u$ be a vertex of $\partial^V Q$. Then $u\not \in V(P)$, hence $u$ lies in the interior of $P$. Consider a face $f$ of $G$ incident with both $u$ and some vertex $v\in V(Q)$. Assume for a contradiction that $f$ is not incident with $P$. Then $v$ belongs to some $\partial^V P \setminus \mathcal{C}$ and we can find a path along the boundary of $f$ that connects $v$ to $u$. Moreover, the interior of $Q$ is connected and so, there is a path in the latter connecting $u$ with $\pazocal{C}$. Hence there is a path in the interior of $P$ connecting $\partial^V P \setminus \mathcal{C}$ to $\pazocal{C}$, which is absurd. Thus $f$ is incident with $P$, which implies that $u\in \partial^V P\cap \partial^V Q\subset \pazocal{C}$, hence $\partial^V Q\subset \pazocal{C}$. This completes the proof of \eqref{inter claim}.
\end{proof}

We are now ready to prove \Tr{main lemma}.

\begin{proof}[Proof of \Tr{main lemma}]
We will first prove the assertion for two special kinds of SAPs which will serve as building blocks for arbitrary SAPs, namely those with $\partial^V P=\emptyset$ and those with $\operatorname{ch}(P)=\emptyset$.
\vspace{.3cm}

\textbf{Case 1: $\partial^V P=\emptyset$.}
\Lr{edges enum} implies that 
\begin{align}\label{edge set}
(k-2)|\operatorname{ch}(P)|=|P|-k,
\end{align} 
as all edges in $\operatorname{ch}(P)$ are counted twice in $m$. This verifies the statement of the lemma in this special case.

\vspace{.3cm}

\textbf{Case 2: $\operatorname{ch}(P)=\emptyset$.} 
Let $\pazocal{C}_1,\pazocal{C}_2,\ldots,\pazocal{C}_N$ be the components of $\partial^V P$ and denote the corresponding SAPs given by \Lr{useful lemma} by $Q_1,Q_2,\ldots, Q_N$, respectively. We want to apply \Lr{upper bound} to every $Q_i$ and for that we need to estimate the total boundary length of each $\pazocal{C}_i$. This is at least $|E(\Gamma[\pazocal{C}_i]^U)|$, and since $\Gamma[\pazocal{C}_i]^U$ is connected with vertex set $\pazocal{C}_i$, $|\pazocal{C}_i|-1\leq |E(\Gamma[\pazocal{C}_i]^U)|$. Thus the total boundary length of $\pazocal{C}_i$ is at least $|\pazocal{C}_i|-1$. (The $-1$ term is only needed in the case $\pazocal{C}_i$ is a singleton.)

Applying \Lr{upper bound} we obtain that $|\pazocal{C}_i|\leq |Q_i|-2k+1-\big( (d-2)(k-2)-4 \big)|I_i|$, where $I_i$ is the vertex set of the interior of $Q_i$. Since $(d-2)(k-2)>4$ and $|\pazocal{C}_i|\leq |I_i|$, we obtain that $\big( (d-2)(k-2)-3 \big)|\pazocal{C}_i|\leq |Q_i|-2k+1$. Summing over all $i$ we conclude that
$$\big( (d-2)(k-2)-3 \big)\sum_{i=1}^N |\pazocal{C}_i|\leq \sum_{i=1}^N |Q_i|-(2k-1)N.$$ 
Notice that distinct $\pazocal{C}_i \cup \partial \pazocal{C}_i$ might intersect when $N>1$ but intersecting $\pazocal{C}_i \cup \partial \pazocal{C}_i$ share only vertices of common faces. For every $x\in V(P)\cup \partial^V P$, let $n(x)$ be the number of sets $\pazocal{C}_i \cup \partial \pazocal{C}_i$ that $x$ belongs to, and define $s=\sum_{x\in V(P)\cup \partial^V P} \big(n(x)-1 \big)$.
Then we have that $$\sum_{i=1}^N |Q_i|\leq \sum_{x\in V(P)} n(x) +\sum_{ x\in \partial^V P} \big( n(x)-1 \big)=|P|+s.$$ 
For this inequality we used that each vertex of $P$ contributes at most $n(x)$ to $\sum_{i=1}^N |Q_i|$, while each vertex of $\partial^V P$ contributes at most $n(x)-1$ because every $\pazocal{C}_i$ is disjoint from $Q_i$ by \eqref{inter claim}.
Notice also that $\sum_{i=1}^N |\pazocal{C}_i|=|\partial^V P|$ because the components $\pazocal{C}_i$ are disjoint. Therefore,
\begin{align}\label{sum}
\big( (d-2)(k-2)-3 \big)|\partial^V P|\leq |P|-(2k-1)N+s,
\end{align}

We claim that 
\begin{align}\label{s formula}
s\leq k(N-1).
\end{align}
Indeed, let $\pazocal{F}$ be the set of faces shared by more than one $\pazocal{C}_i \cup \partial \pazocal{C}_i$ and let $\pazocal{H}$ be the collection of all $\pazocal{C}_i \cup \partial \pazocal{C}_i$. Consider the auxiliary graph $A$ obtained by making the elements of $\pazocal{F}\cup \pazocal{H}$ vertices and connecting a face and some $\pazocal{C}_i \cup \partial \pazocal{C}_i$ whenever the vertices of the face lie in $\pazocal{C}_i \cup \partial \pazocal{C}_i$. 
It follows that $s=\sum_{x\in \pazocal{F}} k\big(d(x)-1\big)$, where $d(x)$ denotes the degree of $x$ in $A$, because each face contains $k$ vertices and each of these vertices is counted in $s$ exactly $d(x)-1$ times. 

Arguing as in the proof of \Lr{upper bound}, we will prove that $A$ is acyclic, i.e.\ it has no cycles. Indeed, let us assume to the contrary that $A$ has a cycle $C$. As we walk around $C$, its vertices alternate between $\pazocal{F}$ and $\pazocal{H}$. As each $\pazocal{C}_i$ is connected, we can connect consecutive faces in $V(C)$ with SAWs in $\Gamma_i$. Each face $f\in \pazocal{F}\cap V(C)$ is incident with exactly two of these SAWs. We connect these two SAWs with a simple curve which, except at its endpoints, lies entirely in the interior of $f$. See Figure \ref{cycle}. In this way we obtain a planar curve $\gamma$ and we want to show that some vertex $u$ of $P$ lies in the bounded region of $\mathbb{R}^2\setminus \gamma$. To see this, consider some vertices $x$ and $y$ visited by $\gamma$ that are incident with a face $f\in \pazocal{F}\cap V(C)$. Then $x$ lies in some $\pazocal{C}_i$ and $y$ lies in some $\pazocal{C}_j$ with $j\neq i$. Removing $x$ and $y$ from the boundary of $f$, we obtain two disjoint paths and each of them needs to contain a vertex of $P$ because otherwise, there is a path in the interior of $P$ connecting $x$ to $y$. The two vertices of $P$ lying in distinct paths of $f$, lie also in distinct regions of $\mathbb{R}^2\setminus \gamma$, hence one of them, which we denote $u$, lies in the bounded region. But $\gamma$ lies in the region bounded by $P$, hence $u$ belongs to the interior of $P$, which is absurd. This contradiction shows that $A$ has no cycles.

\begin{figure}
\centering
\includegraphics[width=.5\linewidth]{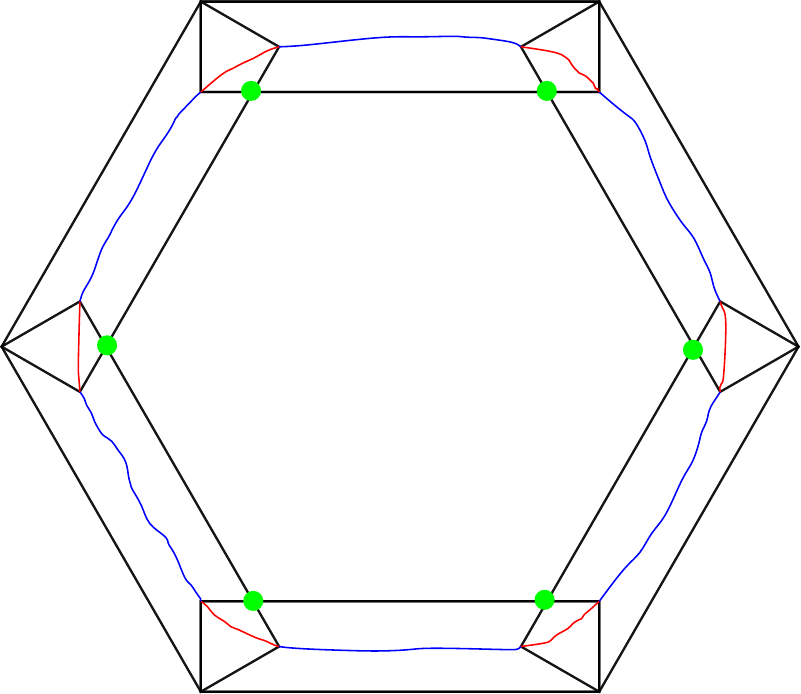}
\caption{The cycle $C$ comprising in this case $6$ SAPs and the curve $\gamma$. The blue parts of $\gamma$ connect consecutive faces. The red ones connect the endvertices of the blue parts and lie inside some face. The green vertices must belong to $P$ as they are incident with distinct components of $\partial^V P$.}
\label{cycle}
\end{figure}

We can now conclude that $|E(A)|\leq |V(A)|-1$ (in fact $A$ is connected and the inequality becomes an equality, but we do not need to use this observation). Moreover, every edge of $A$ is by definition incident with exactly one element of $\pazocal{F}$ and $|V(A)|=|\pazocal{F}|+N$.
Hence $\sum_{x\in \pazocal{F}} \big(d(x)-1 \big) =|E(A)|-|\pazocal{F}|\leq N-1$, which proves \eqref{s formula}. It follows immediately from \eqref{sum} and \eqref{s formula} that
\begin{align}\label{vertex set}
\big( (d-2)(k-2)-3 \big)|\partial^V P|\leq |P|-(k-1)N-k \leq |P|-2k+1.
\end{align}
This verifies the isoperimetric inequality when $\operatorname{ch}(P)=\emptyset$. 

\vspace{.3cm}
\textbf{Case 3: $\partial^V P\neq\emptyset$ and $\operatorname{ch}(P)\neq\emptyset$.} Our aim is to decompose $P$ into some SAPs with empty inner vertex boundary and some SAPs with no inner chords and then use the isoperimetric inequality obtained in the first two cases for each of them.  

We start by constructing an auxiliary graph as follows. We make the connected components of $\partial^V P$ vertices and we connecting two vertices when they are incident with a common face. A \textit{block} of $\partial^V P$ is defined to be a connected component of this auxiliary graph. Consider the set of edges of $\operatorname{ch}(P)$ lying in a face that is incident with $\partial^V P$, and denote this set by $\pazocal{E}$. Let us show that $\pazocal{E}$ is non-empty. Pick an internal face $f$ of $P$ which contains some edge of $\operatorname{ch}(P)$. If $f$ is incident with $\partial^V P$, then we have found an edge of $\pazocal{E}$. If not, then walk around $P$ and consider the sequence $(f_1=f,f_2,\ldots)$ of internal faces of $P$ visited along the way, so that consecutive faces share a common edge (some faces might appear more than once). Eventually we will visit some face which is incident with $\partial^V P$ because $\partial^V P$ is non-empty. Let $f_i$ be the first face of this sequence that is incident with $\partial^V P$. Then $f_i$ shares a common edge $e$ with the previous face $f_{i-1}$. Since $f_{i-1}$ is not incident with $\partial^V P$, $e$ must belong to $\operatorname{ch}(P)$, and since $f_i$ is incident with $\partial^V P$, $e$ must belong to $\pazocal{E}$. 

We can now decompose $\pazocal{E}\cup E(P)$ into some SAPs $P_1,P_2,\ldots,P_r$ with empty inner vertex boundary and some SAPs $Q_1,Q_2,\ldots,Q_n$ with no inner chords, so that the SAPs of this decomposition overlap only on edges of $\pazocal{E}$ and each edge of $\pazocal{E}$ belongs to some $P_i$ or some $Q_i$. Indeed, consider some edge $e\in \pazocal{E}$. Then $P$ contains two edge-disjoint subpaths connecting the endvertices of $e$. Adding $e$ to these subpaths we obtain two SAPs. If $e$ is the only edge in $\pazocal{E}$, then this is our desired decomposition. If not, then we can use an inductive argument to decompose each of the two SAPs and the union of the two decompositions is the desired decomposition of $\pazocal{E}\cup E(P)$. 

It is possible that some edges in $\pazocal{E}$ lie in more than one $Q_i$'s. Let $S$ be the set of those edges and notice that each edge $e\in S$ lies in exactly two $Q_i$'s, one for every face of $G$ incident with $e$. Let also $L$ be the set of edges of $P$ lying in a face incident with a block of $\partial^V P$. Notice that each $Q_i$ is contained in $L\cup\pazocal{E}$. Moreover, for every edge $e'\in L$, there is only one $Q_i$ containing $e'$, since $e'$ is incident with only one internal face of $P$ and by definition of a block, a face can be incident with at most one block of $\partial^V P$. Applying \eqref{vertex set} to every $Q_i$ and summing over all $i$ we obtain 
\begin{equation}
\begin{gathered}\label{arbitrary}
\big( (d-2)(k-2)-3 \big)|\partial^V P|\leq |L|+(|\pazocal{E}|-|S|)+2|S|-(2k-1)n \\ =|L|+|\pazocal{E}|+|S|-(2k-1)n,
\end{gathered}
\end{equation}
because the edges in $L$ and $\pazocal{E}\setminus S$ contribute once to the sum and the edges in $S$ contribute twice to the sum.

Let us now focus on $\operatorname{ch}(P)$. Notice that $\bigcup_{i=1}^r \operatorname{ch}(P_i)= \operatorname{ch}(P)\setminus \pazocal{E}$. Applying \eqref{edge set} to every $P_i$ and summing over all $i$ we obtain
$$(k-2)|\operatorname{ch}(P)\setminus \pazocal{E}|=\sum_{i=1}^r (k-2)|\operatorname{ch}(P_i)|\leq |P|-|L|+|\pazocal{E}|-|S|-kr.$$
For the equality, we used that distinct $P_i$ have disjoint interiors. For the inequality, we used that the edges of each $P_i$ belong to $(E(P)\setminus L)\cup (\pazocal{E}\setminus S)$.
Rearranging we obtain
\begin{align}\label{more complicated ineq}
(k-2)|\operatorname{ch}(P)|\leq |P|-|L|+(k-1)|\pazocal{E}|-|S|-kr.
\end{align}
Combining \eqref{arbitrary} with \eqref{more complicated ineq} we conclude that 
$$(k-2)|\operatorname{ch}(P)|+\big( (d-2)(k-2)-3 \big)|\partial^V P|\leq |P|+k|\pazocal{E}|-(2k-1)n-kr.$$
Finally, we have that $|\pazocal{E}|\leq n+r-1$. Indeed, consider the auxiliary graph obtained by making each $Q_i$ and each $P_j$ a vertex, and connecting two vertices when the corresponding SAPs share a common edge. Let us show that this auxiliary graph has no cycles. Indeed, since the SAPs can only intersect at edges of $\pazocal{E}$, each edge of the new auxiliary graph corresponds to an edge of $\pazocal{E}$. Since every edge of $\pazocal{E}$ has its endvertices in $P$, it divides the region bounded by $P$ into two, and each of these two regions contains some SAP, hence the corresponding edge of the auxiliary graph must be a bridge. Since every edge is a bridge, the auxiliary graph has no cycles. The inequality follows from observing that the number of edges of the auxiliary graph is equal to $|\pazocal{E}|$ and the number of vertices is equal to $n+r$. Therefore, $$(k-2)|\operatorname{ch}(P)|+\big( (d-2)(k-2)-3 \big)|\partial^V P|\leq |P|-(k-1)n-k\leq |P|-k,$$ as desired.
\end{proof}

We will now define a model of mixed percolation that will help us obtain the desired upper bounds for $\mu_p$. Consider some hyperbolic tessellation $\pazocal{H}(d,k)$, and let $q$, $p\in [0,1]$. We first apply site percolation at parameter $q$ on $\pazocal{H}(d,k)$. Edges incident with closed vertices are automatically closed. Then we apply bond percolation at parameter $p$ on the random subgraph of $\pazocal{H}(d,k)$ spanned by the open vertices.

We say that a SAP $P$ \textit{occurs} in a mixed percolation instance $\omega$ if all vertices of $\partial^V P$ and all edges of $\operatorname{ch}(P)$ are closed, and all vertices and edges of $P$ are open.

Using \Lr{main lemma} we obtain the following bounds for $\mu_p$.

\begin{theorem}\label{SAP}
Consider a hyperbolic tessellation $\pazocal{H}(d,k)$, and let $$r=\frac{(k-2)(d-2)-3}{k-2}.$$ Then $\mu_p$ is bounded from above
by the minimum of the function $$\Big((1-p)^{\frac{1}{k-2}}p(1-(1-p)^r)\Big)^{-1}$$ on the interval $[0,1]$.
\end{theorem}
\begin{remark}
Letting $p=\frac{1}{k-1}$ and $d$ tend to infinity, we obtain that $\limsup_{d\to\infty} \mu_p \big(\pazocal{H}(d,k)\big)\leq\Big(\frac{k-1}{k-2}\Big)^{\frac{k-1}{k-2}}$
\end{remark}
\begin{proof}
Let $q=1-(1-p)^r$. The probability that a SAP $P$ of length $n$ occurs is equal to $(1-p)^{|\operatorname{ch}(P)|} (1-p)^{r|\partial^V P|} p^n (1-(1-p)^r)^n.$ Using \Lr{main lemma} we obtain that
\begin{equation} \label{probability}
\begin{gathered}
(1-p)^{|\operatorname{ch}(P)|} (1-p)^{r|\partial^V P|} p^n (1-(1-p)^r)^n\geq \\ \Big((1-p)^{\frac{1}{k-2}}p(1-(1-p)^r)\Big)^n.
\end{gathered}
\end{equation}

We claim that if two distinct SAPs $P_1$ and $P_2$ occur and contain $o$, i.e.\ $o\in V(P_1)$ and $o\in V(P_2)$, then the (topologically) open regions $R_1$ and $R_2$ bounded by them are disjoint. If some vertex of $P_1$ belongs to the interior of $P_2$, then all vertices of $P_1$ belong to the interior of $P_2$, because the vertices of $P_1$ are open, while the vertices of $\partial^V P_2$ are closed. However, this is absurd because $o$ belongs to both $P_1$ and $P_2$. Hence no vertex of $P_1$ belongs to the interior of $P_2$. This implies that all vertices of $P_1$ lie in $\mathbb{R}^2\setminus R_2$. In particular, either $R_2\subset R_1$ or $R_1\cap R_2=\emptyset$. If $R_1\cap R_2\neq \emptyset$, then reversing the roles of $R_1$ and $R_2$, we obtain the reverse containment $R_1\subset R_2$, hence $R_1=R_2$. This implies that $P_1=P_2$, which is absurd. Therefore, we have $R_1\cap R_2=\emptyset$, as desired.

Let now $N_n$ be the number of occurring SAPs of length $n$ that contain $o$.
Since there are $d$ faces incident with $o$ and the open regions bounded by occurring SAPs are disjoint, it follows that $N_n\leq d$ for any percolation instance $\omega$, implying that $E_p(N_n)\leq d$. Moreover, 
$$E_p(N_n)\geq p_n \Big(p^{\frac{1}{k-2}}(1-p)(1-p^r)\Big)^{n}$$ by \eqref{probability}. We conclude that 
\begin{align}\label{generic bound}
p_n\leq d \Big((1-p)^{\frac{1}{k-2}}p(1-(1-p)^r)\Big)^{-n}.
\end{align}
Letting $p$ be the point that minimizes the function $$\Big((1-p)^{\frac{1}{k-2}}p(1-(1-p)^r)\Big)^{-1}$$ on the interval $[0,1]$, we obtain the desired assertion.
\end{proof}

\Tr{SAP} gives strong bounds for all hyperbolic tessellations $\pazocal{H}(d,k)$ except for those with $d=3$. For example, in the case of $\pazocal{H}(3,7)$ and $\pazocal{H}(3,8)$ the bounds we obtain are greater than $2$. In the next theorem we improve the bounds of \Tr{SAP} for every $\pazocal{H}(3,k)$. Let us first recall that the dual graph $\pazocal{H}(3,k)^*$ of $\pazocal{H}(3,k)$ is defined by placing a vertex at the interior of each face of $\pazocal{H}(3,k)$ and connecting two vertices when the corresponding faces share a common primal edge. 

\begin{theorem}\label{SAP2}
For any $\pazocal{H}(3,k)$, $\mu_p$ is bounded from above by $$N_k:=\frac{k-4}{(k-5)^{\frac{k-5}{k-4}}}.$$
\end{theorem}
\begin{proof}
Let $P$ be a SAP of $\pazocal{H}(3,k)$ containing $o$. Consider the internal faces of $P$ that are incident with $P$. We let $M$ be the set of vertices of $\pazocal{H}(3,k)^*$ that are dual to these faces. Consider now the external faces of $P$ that are incident with $P$. We let $B$ be the set of vertices of $\pazocal{H}(3,k)^*$ that are dual to the latter faces.

Let us make a few observations. First, $M$ spans a connected graph since $P$ is connected. Moreover, every face sharing at least one vertex with $P$ shares a common edge with $P$ because, otherwise, the common vertex has degree at least $4$, which is absurd. This implies that every vertex of $B$ is adjacent with some vertex of $M$. Hence the pair $(M,B)$ is an \pint\ in $\pazocal{H}(3,k)^*$. Notice also that \labtequ{dual face}{\centering the subgraph of $\pazocal{H}(3,k)^*$ spanned by $B$ contains a SAP such that the dual face of $o$ lies in the region bounded by this SAP and is incident with it.} 

Given some $r>0$, we let $S_{n,r}$ denote the set of all \pint s $(M,B)$ produced in this way, such that $|M|=n$ and $|B|=rn$, and we let $\pazocal{M}_{n,r}=|S_{n,r}|$. Notice that \labtequ{zero}{\centering $\pazocal{M}_{n,r}=0$ for every $r>k$,} since every vertex of $M$ is incident with at most $k$ vertices of $B$ and every vertex of $B$ is incident with some vertex of $M$. Write $p_{m,n,r}$ for the number of all SAPs of $\pazocal{H}(3,k)$ with $m$ edges for which the corresponding \pint\ $(M,B)$ lies in $S_{n,r}$. Our aim is to find an upper bound for $\pazocal{M}_{n,r}$ in terms of $n$, and then upper bound $n$ in terms of $m$.

We will now do site percolation on $\pazocal{H}(3,k)^*$. We say that an \pint\ occurs in a site percolation instance $\omega$, if the vertices of $B$ are closed, and the vertices of $M$ are open. It is not hard to see that at most one element of $S_{n,r}$ occurs in any $\omega$, since occurring \pint s are disjoint (see \cite{GeoPaAnalytic}) and \eqref{dual face} holds. Arguing as in the proof of \Tr{SAP}, we obtain
$$\pazocal{M}_{n,r}\leq \big(p(1-p)^r\big)^{-n}$$ for any $p\in [0,1]$. Letting $p=1/(1+r)$, we conclude that
$$\pazocal{M}_{n,r}\leq \frac{(r+1)^{(r+1)n}}{r^{rn}},$$ 
hence 
\begin{align}\label{ineq}
p_{m,n,r}\leq \frac{(r+1)^{(r+1)n}}{r^{rn}}
\end{align}
as well.

\begin{figure}
\begin{center}
\includegraphics[width=0.4\textwidth]{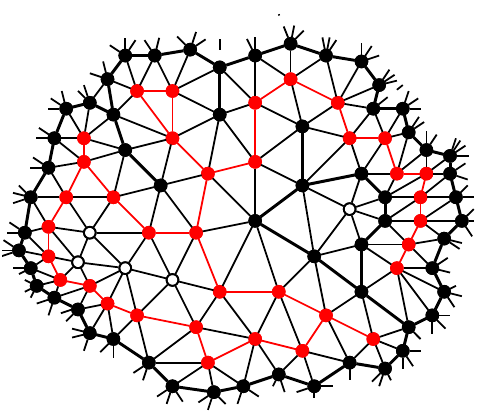}\hspace{.10\textwidth}\includegraphics[width=0.4\textwidth]{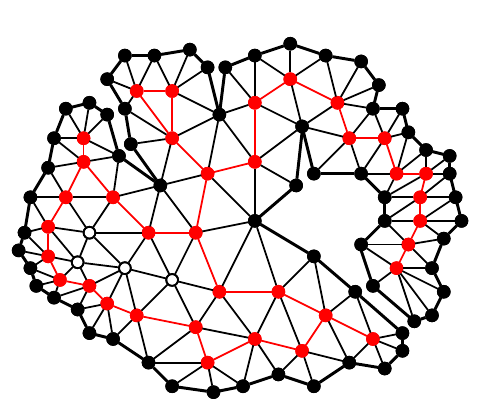}
\end{center}
\caption{Unzipping $B$ (bold vertices and edges); $M$ is shown in red (if colour is shown).}\label{fig}
\end{figure}

In \cite{SitePercoPlane} an `unzipping' operation is defined that turns $B$ into a SAP $Q$. As we will see, $Q$ lives, in general, on a different graph. To apply this operation, we will use that $\pazocal{H}(3,k)^*$ is a planar triangulation ($\pazocal{H}(3,k)^*$ is isomorphic to $\pazocal{H}(k,3)$). This ensures that the subgraph of $\pazocal{H}(3,k)^*$ spanned by $B$ is a connected graph that contains a SAP $R$ with the property that every vertex of $B$ lies either in $R$ or in its interior. Using the connectivity of $M$ and the fact that every vertex of $B$ is adjacent to some vertex of $M$, we can deduce that all vertices of $B$ lie at the boundary of a bounded region of $\mathbb{R}^2$. Let us briefly describe this operation. We imagine that each edge spanned by $B$ has positive width so that each such edge has two edge-sides, where each of them is incident with exactly one face. Moreover, each edge-side has two ends reaching the endvertices of the corresponding edge. Follow $B$ clockwise,  writing down a list of vertices visited (so the same vertex can appear in the list multiple times). Record also the ends of edges between $M$ and $B$ which are crossed in a cyclic ordering, and group
these edge-ends by the vertex in $B$ which they reach. We now `unzip' $B$ by replacing vertices in $B$ by the entries of the list, so that each vertex
which appears more than once in the list is split into multiple vertices distinguished by list position. We also replace the edges spanned by the vertices in $B$ by edges between consecutive entries in the list. In this way, we obtain a SAP $Q$. There is an one-to-one correspondence between groups of edge-ends and entries in the list; we use this correspondence to replace every edge between $M$ and $B$ by an edge between $M$ and a specific list entry. 
Figure~\ref{fig} illustrates this unzipping operation.\footnote{I thank John Haslegrave for creating Figure~\ref{fig}.} It is clear that this operation preserves the vertex and face degrees of all vertices and faces in the region bounded by $B$ (here we exclude all vertices of $B$). In particular, the number of edges between $Q$ and $M$ is the same as the number edges between $B$ and $M$, which equals $m$. Let us consider the graph induced by $Q$ and the interior of $Q$. This graph is not, in general, a subgraph of $\pazocal{H}(3,k)^*$ but it belongs to the general class of graphs $\pazocal{S}(k,3)$ we are working on. \\ \indent
Notice that $\partial^V Q$ coincides with $M$ because every vertex of $M$ is adjacent to some vertex of $Q$ and every vertex of $Q$ is adjacent to some vertex of $M$. Let us show that $Q$ has no inner chords. Indeed, if $Q$ has an inner chord $e$, then $e$ divides the region bounded by $Q$ into two. Since $M$ is connected, one of these two regions contains $M$ entirely in its interior, so the other region contains a vertex of $Q$ not adjacent to $M$, which is absurd. \\ \indent
Recall that $m$ is equal to the number of edges between $Q$ and $M$.
Letting $n'$ denote the total boundary length of $\partial^V Q$ and $I$ be the set of vertices in the interior of $Q$, we can now apply \Lr{upper bound} to obtain that $m=|Q|+n'$ and $n'=|Q|-6-(k-6)|I|$. Since $n'\geq |M|-1$ and $|I|\geq |M|$, we obtain that 
$m\geq |Q|+|M|-1$ from the first equality
and $|Q|\geq (k-5)|M|$ from the second inequality. Combining the latter inequalities we obtain that $m\geq (k-4)|M|-1$. Moreover, it is clear from the construction that $|Q|\geq |B|$,
hence 
\begin{equation}\label{equat 1}
m\geq |B|+|M|-1.
\end{equation} 
Therefore, $m\geq \max\{(r+1)n,(k-4)n\}-1$ for any of the $p_{m,n,r}$ SAPs, which combined with \eqref{ineq} implies that 
$$p_{m,n,r}\leq \frac{(r+1)^{m+1}}{r^{\frac{r}{r+1}(m+1)}} \quad \text{ and } \quad p_{m,n,r}\leq\frac{(r+1)^{\frac{r+1}{k-4}(m+1)}}{r^{\frac{r}{k-4}(m+1)}}.$$
Since $p_{m,n,r}=0$ for every $r>k$ by \eqref{zero}, we obtain that
$$p_{m,n,r}\leq C\min \left \{\frac{(r+1)^m}{r^{\frac{r}{r+1}m}},\frac{(r+1)^{\frac{r+1}{k-4}m}}{r^{\frac{r}{k-4}m}}\right\}$$
for some constant $C>0$ by observing that the functions $g(r):=\frac{r+1}{r^{\frac{r}{r+1}}}$ and $h(r):=\frac{(r+1)^{\frac{r+1}{k-4}}}{r^{\frac{r}{k-4}}}$ are bounded for $0<r\leq k$. Notice that $g(k-5)=h(k-5)=N_k$. Since $g$ is decreasing for any $r\geq 1$, and $h$ is increasing for any $r>0$, it follows that $p_{m,n,r}\leq CN^m _k$ for every $r>0$. 

To obtain an upper bound for $p_m$, we notice that $|B|\leq m$ and $|M|\leq m$ by \eqref{equat 1}. Hence, there are at most $m^2$ possibilities for the triples $(m,n,r)$ with $|P|=m$, $|B|=n$ and $|M|=rn$ for any fixed $m$, which implies that 
\begin{align}\label{degree 3}
p_m\leq Cm^2 N^m _k,
\end{align}
The desired assertion follows immediately. 
\end{proof}

We stress that the function $$\dfrac{(1+r)^{1+r}}{r^r}$$ appearing in the above proof is a universal upper bound for the exponential growth rate of the number of SAPs of `surface-to-volume ratio $r$', independent of the underlying graph. The function features also in \cite{ExpGrowth, HammondExpRates}.

\section{Lower bounds for $\mu_w$}\label{sec-l}

The aim of this section is to prove \Tr{main theorem} and \Tr{asymptotics}. To this end, we first need to introduce a few definitions.

Consider some hyperbolic tessellation  $\pazocal{H}(d,k)$ and let $o$ be one of its vertices. For every $n\geq 1$, we define the $n$th layer of the graph as follows. The first layer consists of those vertices except $o$ that lie in a face incident with $o$. The second layer consists of those vertices except $o$ or the vertices of the first layer that lie in a face incident with the first layer. The other layers can be defined inductively. For convenience, we define the $0$th layer to be simply $o$. Let us write $G_n$ for the graph spanned by the $n$th layer. It will be important later on that each $G_n$ is a cycle, which we believe is known to the experts and intuitively clear. A proof of this fact is included in \Lr{span cycle} in the Appendix.

Recall that the $n$th layer of the dual graph $\pazocal{H}(3,k)^*$ of $\pazocal{H}(3,k)$ was defined in Section \ref{correction} as follows. Consider the faces of $\pazocal{H}(3,k)$ between layers $n$ and $n+1$. Then the $n$th layer of $\pazocal{H}(3,k)^*$ is the set of the corresponding dual vertices. We remark that the following lemmas can be proved mutatis mutandis for the $n$th layer of $\pazocal{H}(3,k)^*$. They are stated under the assumption that $G_n$ is a SAP because they will be used in the proof of \Lr{span cycle}. 

\begin{lemma}\label{neigh-3}
Consider some hyperbolic triangulation $\pazocal{H}(d,3)$. Assume that for some $n\geq 1$ and every $1\leq i\leq n$, each $G_i$ is a SAP that contains all previous layers in its interior. Let $u$ be a vertex of $G_n$. Then $u$ has either $1$ or $2$ neighbours in its previous layer, hence at least $d-4$ next layer neighbours.
\end{lemma}
\begin{proof}
Let us assume to the contrary that $u$ has at least $3$ neighbours in $G_{n-1}$. Since $G_n$ is a SAP and $G_{n-1}$ lies in its interior, the neighbours of $u$ in $G_{n-1}$ span a subpath of $G_{n-1}$. Hence $u$ must have $3$ neighbours $x,y,z\in G_{n-1}$ so that $y$ is connected to both $x$ and $z$. Then $y$ has only $1$ next layer neighbour, namely $u$. As $y$ has $2$ same layer neighbours and has degree $d\geq 7$, it has at least $3$ neighbours in its previous layer (in fact at least $4$). Continuing in this manner we deduce that some vertex of the first layer has at least $3$ neighbours in its previous layer. But this is a contradiction, as all vertices in the first layer have exactly $1$ previous layer neighbour. Hence $u$ has either $1$ or $2$ neighbours in its previous layer.
\end{proof}

\begin{lemma}\label{neigh-k}
Consider some hyperbolic tessellation $\pazocal{H}(d,k)$ with $k>3$. Assume that for some $n\geq 1$ and every $1\leq i\leq n$, each $G_i$ is a SAP that contains all previous layers in its interior. Let $u$ be a vertex of $G_n$. Then $u$ has either $0$ or $1$ neighbours in its previous layer, hence at least $d-3$ next layer neighbours.
\end{lemma}
\begin{proof}
Let us assume to the contrary that $u$ has at least $2$ neighbours in its previous layer. Since $G_n$ is a SAP and $G_{n-1}$ lies in its interior, there is a face containing $u$ and two neighbours $x$ and $y$ of $u$ which belong to $G_{n-1}$. Since $k>3$, this face needs to contain at least one more vertex $z$, which necessarily belongs to the same layer as $x$ and $y$. Now $z$ has no next layer neighbours and $2$ same layer neighbours, hence it has at least $2$ previous layer neighbours. Continuing in this fashion, we find that a vertex at the first layer has at least $2$ previous layer neighbours. But this is a contradiction as all vertices in the first layer have at most $1$ previous layer neighbour. Hence $u$ has either $0$ or $1$ neighbours in its previous layer.
\end{proof}

In the next lemma we prove that the length of the $n$th layer grows exponentially in $n$. 

\begin{lemma}\label{exp growth}
Consider some hyperbolic tessellation $\pazocal{H}(d,k)$. Then the length of the $n$th layer is at least $\big((d-2)(k-2)-3\big)^n$.
\end{lemma}
\begin{proof}
Let $L_n$ denote the length of the $n$th layer. It follows from \Lr{main lemma} that $L_n\geq \big((d-2)(k-2)-3\big) L_{n-1}$. Iterating this inequality and using that the inner vertex boundary of the first layer has size $1$, we obtain the desired result.
\end{proof}

For every $\pazocal{H}(d,k)$, let $\pazocal{M}=\pazocal{M}(d,k)$ denote the minimum of the upper bounds on $\mu_p$ of \Tr{SAP} and \Tr{SAP2}, and $\pazocal{R}=\pazocal{R}(d,k)$ denote the lower bound on $\mu_w$ of Proposition~\ref{prop corrected} and Proposition~\ref{lower bounds}. In the following table, we gather the approximate values of $\pazocal{M}$ and $\pazocal{R}$ for the tessellations $\pazocal{H}(d,k)$ not in $\pazocal{L}$ for which $\pazocal{M}<\pazocal{R}$.

\vspace{.5cm}

\begin{center}
 \begin{tabular}{|c |c |c |c |c |c|} 
 \hline
  & $\pazocal{H}(8,3)$ & $\pazocal{H}(9,3)$ & $\pazocal{H}(5,4)$ & $\pazocal{H}(3,9)$ & $\pazocal{H}(3,10)$ \\
  \hline
  $\pazocal{M}$ & $4.43218$ & $4.21306$ & $3.04321$ & $1.64938$ & $1.56919$\\
  \hline
  $\pazocal{R}$ & $5.47722$ & $6.48074$ & $3.30192$ & $1.74110$ & $1.78179$\\
 \hline
\end{tabular}
\end{center}

\vspace{.5cm}

We are now ready to prove \Tr{main theorem}.

\begin{proof}[Proof of \Tr{main theorem}]
Comparing the bounds on the above table, we see that $\mu_p<\mu_w$ for $\pazocal{H}(8,3)$, $\pazocal{H}(9,3)$, $\pazocal{H}(5,4)$, $\pazocal{H}(3,9)$ and $\pazocal{H}(3,10)$. We also have $\mu_p<\mu_w$ for the class $\pazocal{L}$. It remains to prove the assertion for $\pazocal{H}(7,3)$, $\pazocal{H}(4,5)$, $\pazocal{H}(3,7)$ and $\pazocal{H}(3,8)$.

Let us start with $\pazocal{H}(7,3)$. We will study the families $W(x)$ of SAWs that start from some vertex $x$ and are not allowed to move to a previous layer. For every vertex $x$, we denote $W_n(x)$ the SAWs of $W(x)$ of length $n$, and for every vertex $x\neq o$, denote $W^+ _n(x)$ (respectively $W^- _n(x)$) the SAWs of $W_n(x)$ which are not allowed to visit at any step, the same layer neighbour of $x$ on the anticlockwise (resp. clockwise) direction. 

We will partition the vertices of $\pazocal{H}(7,3)$ into $5$ sets $S_i$ according to the structure of their neighbourhoods. The first set contains only $o$. The second set contains all those vertices which have $3$ neighbours in their next layer. The remaining vertices have exactly $4$ next layer neighbours by \Lr{span cycle} and \Lr{neigh-3}. The third set contains those vertices (not in $S_1$ or $S_2$) with the property that both same layer neighbours have $4$ neighbours in their next layer (hence only the vertices of the first layer can belong to this set). The fourth set contains those vertices with the property that one of the same layer neighbours has $3$ next layer neighbours and the other has $4$ next layer neighbours. Finally, the fifth set contains those vertices with the property that both same layer neighbours have $3$ neighbours in their next layer. 

The above observations can be used to express $W_n(x)$ using the following recurrences:
$$|W_n(x)|= |W_{n-1}(x_1)|+|W_{n-1}(x_2)|+|W_{n-1}(x_3)|+|W^+ _{n-1}(x^+)|+|W^- _{n-1}(x^-)|$$
for every $x\in S_2$, and
\begin{gather*}
|W_n(x)|=|W_{n-1}(x_1)|+|W_{n-1}(x_2)|+|W_{n-1}(x_3)|+|W_{n-1}(x_4)|+\\|W^+ _{n-1}(x^+)|+|W^- _{n-1}(x^-)|
\end{gather*}
for every $x\in S_3,S_4 \text{ or } S_5$, where $x_1,x_2,x_3$ (and $x_4$) denote the neighbours of $x$ in its next layer, and $x^+,x^-$ denote the neighbours of $x$ on the clockwise and anticlockwise direction, respectively (recall that when defining $W^+ _{n-1}(x^+)$ we are not allowed to visit $x$ at any step). Moreover, if $n< l(x)$, where $l(x)$ denotes the length of the layer that $x$ belongs to, then
$$|W^+_n(x)|= |W_{n-1}(x_1)|+|W_{n-1}(x_2)|+|W_{n-1}(x_3)|+|W^+ _{n-1}(x^+)|$$
for every $x\in S_2$, and
$$|W^+_n(x)|=|W_{n-1}(x_1)|+|W_{n-1}(x_2)|+|W_{n-1}(x_3)|+|W_{n-1}(x_4)|+|W^+ _{n-1}(x^+)|$$
for every $x\in S_3,S_4 \text{ or } S_5$. Similar recurrence relations are valid for $|W^-_n(x)|$. Analysing these recurrence relations seems unnecessarily hard, as we only need a lower bound for $\mu_w$. Instead, we will compare these recurrence relations with another system of recurrence relations that is easier to analyse in order to find some lower bounds for $|W_n(x)|$. 

It is natural to expect that among vertices $x$ of the same layer, $|W_n(x)|$ is minimized when $x$ has $3$ next layer neighbours. Moreover, if $u\in S_4$ and $v\in S_5$ are vertices of the same layer, then we expect that $|W_n(u)|\geq |W_n(v)|$. With these considerations in mind we introduce four sequences $a_n,b_n,c_n$ and $d_n$ satisfying the following recurrence relations:
\vspace{-0.7cm}
\begin{multicols}{2}
\[a_n:= \begin{cases} 
      5 & n=1 \\
      2a_{n-1}+b_{n-1}+2d_{n-1} & n\geq 2,
   \end{cases}
\] \break
\[b_n:= \begin{cases} 
      6 & n=1 \\
      2a_{n-1}+2b_{n-1}+2c_{n-1} & n\geq 2,
   \end{cases}
\]
\end{multicols}
\begin{multicols}{2}
\[c_n:= \begin{cases} 
      4 & n=1 \\
      2a_{n-1}+b_{n-1}+d_{n-1} & n\geq 2,
   \end{cases}
\] 
\vfill
\columnbreak 
\vspace*{\fill}
\[d_n:= \begin{cases} 
      5 & n=1 \\
      2a_{n-1}+2b_{n-1}+c_{n-1} & n\geq 2.
   \end{cases}
\]
\end{multicols}
One can come up with those relations by considering a SAW in $W_n(x)$ for some $x$ in $S_2$ or $S_5$, and each time our SAW visits a vertex in $S_4$, `pretend' that it will move in its next step as if it was at a vertex in $S_5$. The sequences $a_n$ and $b_n$ correspond to $|W_n(x)|$ for $x$ in $S_2$ and $S_5$, respectively, while $c_n$ and $d_n$ correspond to $|W^+ _n(x)|$, $|W^- _n(x)|$ for $x$ in $S_2$ and $S_5$, respectively. Notice that it is impossible for $2$ neighbouring vertices in the same layer to both belong to $S_2$, because, otherwise, their common previous layer neighbour has only $2$ next layer neighbours. Thus both same layer neighbours of a vertex $u\in S_2$ lie either in $S_4$ or $S_5$. This explains why the term $d_{n-1}$ appears in the recurrence relations of $a_n$ and $c_n$.

We claim that 
\begin{align}\label{main claim}
|W_{n}(x)|\geq a_n
\end{align} 
for every $x\neq o$ and $n< l(x)$. Indeed, the recurrences imply that 
\begin{align}\label{intermediate}
d_n\leq b_n. 
\end{align}
It follows inductively from the latter inequality that 
\begin{align}\label{first key}
c_n\leq d_n,
\end{align} 
which in turn shows that $b_n\leq 2c_n$. Combining the latter inequality with \eqref{intermediate} we obtain 
\begin{align}\label{second key}
a_n\leq b_n.
\end{align}
Comparing the recurrence relations satisfied by $|W_n(x)|$, $|W^+_n(x)|$, $|W^-_n(x)|$ with those satisfied by $a_n$, $b_n$, $c_n$, $d_n$,
and using \eqref{first key}, \eqref{second key}, we can easily see inductively  that:
\begin{enumerate}
\item $|W_n(x)|\geq a_n$ for every vertex $x$ and any $n< l(x)$,
\item $|W_n(x)|\geq b_n$ for every $x\in S_3,S_4$ or $S_5$ and any $n< l(x)$,
\item $|W^+ _n(x)|,|W^- _n(x)|\geq c_n$ for every vertex $x$ and any $n< l(x)$,
\item $|W^+ _n(x)|,|W^- _n(x)|\geq d_n$ for every $x\in S_3\cup S_4\cup S_5$ and any $n< l(x)$.
\end{enumerate}
The first item verifies \eqref{main claim}. 

Notice that we can extend any SAW of $W(x)$ to a SAW of $W(o)$ by adding a geodesic from $o$ to $x$. Applying \eqref{main claim} we obtain $W_n(o)\geq W_m(x)$, where $m=l(x)$ and $n=m+d(o,x)$. Notice that $d(o,x)=o\big(l(x)\big)$ because by \Lr{exp growth}, the length of the $i$th layer grows exponentially in $i$, while the distance of any vertex of the $i$th layer from $o$ is $i$.
Thus we obtain that $$\limsup_{n\to \infty} {|W_n(o)|}^{1/n} \geq \liminf_{n\to \infty} {a_n}^{1/n}.$$ 
Using standard arguments we can check that $a_n \sim A \lambda^n$, where $A>0$ is a constant and $\lambda\approx 5.13912$ is the largest eigenvalue of the matrix
$$\begin{bmatrix}
2 & 1 & 0 & 2 \\
2 & 2 & 2 & 0 \\
2 & 1 & 0 & 1 \\
2 & 2 & 1 & 0
\end{bmatrix}.$$
On the other hand, $\mu_p\leq \pazocal{M}\approx 4.9575$ by \Tr{SAP}. Hence $\mu_p<\mu_w$.

We will use a similar strategy for the remaining hyperbolic tilings. The families $W(x)$, $W_n(x)$, $W^+ _n(x)$ and $W^- _n(x)$ are defined analogously.
Once again, we partition the vertices of the tessellation into sets according to their neighbourhoods. After a worst case analysis we are led to a system of recursive relations. The largest eigenvalue of the corresponding matrix is a lower bound for $\mu_w$.

In the case of $\pazocal{H}(4,5)$, we partition the vertices into the following sets. The first set comprises $o$. The second set comprises those vertices that have $1$ next layer neighbour. The remaining vertices have $2$ next layer neighbours by \Lr{span cycle} and \Lr{neigh-k}. The third set comprises those vertices that have $2$ next layer neighbours and $2$ same layer neighbours lying in the second set. The fourth set consists of the remaining vertices, i.e.\ vertices with $2$ next layer neighbours and at least $1$ same layer neighbour in  $S_3$ or $S_4$.

If each time our walks visit a vertex in $S_4$ we pretend that they will move as if they were at vertex in $S_3$, then we can come up with the following recursive relations:
\vspace{-0.8cm}
\begin{multicols}{2}
\[a_n:= \begin{cases} 
      3 & n=1 \\
      a_{n-1}+2c_{n-1} & n\geq 2,
   \end{cases}
\] 
\vfill
\columnbreak 
\vspace*{\fill}
\[b_n:= \begin{cases} 
      2 & n=1 \\
      a_{n-1}+c_{n-1} & n\geq 2,
   \end{cases}
\]
\end{multicols}
\[c_n:= \begin{cases} 
      3 & n=1 \\
      2a_{n-1}+b_{n-1} & n\geq 2.
   \end{cases}
\]
Now $a_n$ and $b_n$ correspond to $|W_n(x)|$ and $|W^+ _n(x)|$ for $x\in S_2$, while $c_n$ corresponds to $|W^+ _n(x)|$ for $x\in S_3$. Let us mention that for every vertex $u$ in $S_2$, both same layer neighbours lie in $S_3$ or $S_4$, which explains why the term $c_{n-1}$ appears in the recurrence relations for $a_n$ and $b_n$. To see this, assume that $u$ and a same layer neighbour $v$ of $u$ belong to $S_2$. Write $x$ and $y$ for the previous layer neighbours of $u$ and $v$. Then the face that contains all these vertices contains also a fifth vertex $z$. But $z$ has now no next layer neighbours which is a contradiction.

Notice that $c_n\geq b_n$. In fact, this follows from the stronger statement that both $c_n\geq b_n$ and $a_n+b_n\geq c_n$ hold simultaneously, which can be proved inductively. Arguing as in the case of $\pazocal{H}(7,3)$, we see that $|W_n(x)|\geq a_n$ and $|W^+ _n(x)|\geq b_n$ for every $x\in S_2$, while $|W^+_n(x)|\geq c_n$ for every $x\in S_3\cup S_4$. Hence $\mu_w$ is not smaller than the largest eigenvalue of the corresponding matrix, which is approximately $2.86619$. On the other hand, \Tr{SAP} gives that $\mu_p\leq \pazocal{M}\approx 2.60371$. This proves that $\mu_p<\mu_w$, as desired.

In the case of $\pazocal{H}(3,7)$, given a vertex $x$ lying at a layer $n$ and being incident to layer $n-1$, let $x+i$ and $x-i$ denote the $i$th vertex along the same layer on $x$ on the clockwise, anticlockwise direction, respectively. We claim that for exactly one of the following pairs, both vertices are incident to layer $n-1$: $(x+4,x-4)$, $(x+4,x-3)$, $(x+3,x-4)$. Indeed, consider the dual  graph $\pazocal{H}(3,7)^*$. Each vertex of $\pazocal{H}(3,7)^*$ not in the first layer has $3$ or $4$ next layer neighbours and furthermore, for any pair of neighbouring vertices lying in the same layer, at least one of them has $4$ next layer neighbours. This easily implies the claim.

If whenever our walks visit a new layer at a vertex $x$ we pretend that either both vertices of the pair $(x+4,x-3)$ or both vertices of the pair $(x+3,x-4)$ are incident to the previous layer, then we can obtain the following recurrence relations:
\vspace{-0.8cm}
\begin{multicols}{2}
\[a_{n,0}:= \begin{cases} 
      2 & n=1 \\
      a_{n-1,1}+a_{n-1,-1} & n\geq 2,
   \end{cases}
\]
\vfill
\columnbreak 
\vspace*{\fill}
\[a_{n,1}:= \begin{cases} 
      2 & n=1 \\
      a_{n-1,0}+a_{n-1,2} & n\geq 2.
   \end{cases}
\]
\end{multicols}
\begin{multicols}{2}
\[a_{n,2}:= \begin{cases} 
      2 & n=1 \\
      a_{n-1,0}+a_{n-1,3} & n\geq 2,
   \end{cases}
\]
\vfill
\columnbreak 
\vspace*{\fill}
\[a_{n,3}:= \begin{cases} 
      1 & n=1 \\
      a_{n-1,-1} & n\geq 2,
   \end{cases}
\]
\end{multicols}
\begin{multicols}{2}
\[a_{n,-1}:= \begin{cases} 
      2 & n=1 \\
      a_{n-1,0}+a_{n-1,-2} & n\geq 2.
   \end{cases}
\]
\vfill
\columnbreak 
\vspace*{\fill}
\[a_{n,-2}:= \begin{cases} 
      2 & n=1 \\
      a_{n-1,0}+a_{n-1,-3} & n\geq 2,
   \end{cases}
\]
\end{multicols}
\begin{multicols}{2}
\[a_{n,-3}:= \begin{cases} 
      2 & n=1 \\
      a_{n-1,0}+a_{n-1,-4} & n\geq 2,
   \end{cases}
\]
\vfill
\columnbreak 
\vspace*{\fill}
\[a_{n,-4}:= \begin{cases} 
      1 & n=1 \\
      a_{n-1,1} & n\geq 2,
   \end{cases}
\]
\end{multicols}
Here $a_{n,0}$ corresponds to $|W_n(x)|$ for a vertex $x$ with no next layer neighbours and with both $x+3$ and $x-4$ having a previous layer neighbour. Moreover, for $i=1,2,3$, $a_{n,i}$ corresponds to $|W^+_n(x+i)|$, while for $i=1,2,3,4$, $a_{n,-i}$ corresponds to $|W^-_n(x-i)|$. We claim that $a_{n,-1}\geq a_{n,1}$, $a_{n,-2}\geq a_{n,2}$, $a_{n,-3}\geq a_{n,3}$ and $a_{n,0}\geq a_{n,-1}$. Indeed, first, it is easy to see inductively that for every $n\geq 1$ and every $i=0,\pm 1, \pm 2, \pm 3, 4$, we have $a_{n+1,i}\leq 2 a_{n,i}$. Now using the recurrence relations and dropping some positive terms, we obtain that
$$a_{n,-3}=a_{n-2,-1}+2a_{n-2,1}\geq 3a_{n-3,0}+2a_{n-3,2}\geq 3a_{n-3,0}+2a_{n-4,0}.$$
Since $2a_{n-3,0}\geq a_{n-2,0}$, we deduce that $a_{n,-3}\geq a_{n-2,0}+a_{n-3,0}+2a_{n-4,0}$. On the other hand, $$a_{n,3}=a_{n-2,0}+a_{n-3,0}+a_{n-4,0}+a_{n-4,-4}\leq a_{n-2,0}+a_{n-3,0}+2a_{n-4,0},$$
where in the last inequality we used that $a_{n-4,-4}\leq a_{n-4,0}$. This proves that $a_{n,-3}\geq a_{n,3}$. It now follows from the recurrence relations that $a_{n,-1}\geq a_{n,1}$ and $a_{n,-2}\geq a_{n,2}$. Finally, we have that
$$a_{n,-3}=a_{n-1,0}+a_{n-1,-4}\geq 2a_{n-1,-4}\geq a_{n,-4},$$ which combined with the recurrences implies in turn that $a_{n,-2}\geq a_{n,-3}$, $a_{n,-1}\geq a_{n,-2}$ and $a_{n,0}\geq a_{n,-1}$ and proves the claim.

We can now argue as above and use the claim to deduce that for every vertex $x$ incident to the previous layer, 
\begin{enumerate}
\item $|W_n(x)|\geq a_{n,0}$,
\item if both $x+3$ and $x-4$ are incident to the previous layer, then $|W^+_n(x+i)|\geq a_{n,i}$ for every $i=1, 2, 3$ and $|W^-_n(x-j)|\geq a_{n,-j}$ for every  $j=1, 2, 3,4$,
\item if both $x+4$ and $x-3$ are incident to the previous layer, then $|W^+_n(x+i)|\geq a_{n,-i}$ for every $i=1, 2, 3,4$ and $|W^-_n(x-j)|\geq a_{n,j}$ for every  $j=1, 2, 3$,
\item if both $x+4$ and $x-4$ are incident to the previous layer, then $|W^+_n(x+i)|,|W^-_n(x-i)| \geq a_{n,-i}$ for every $i=1,2,3,4$.
\end{enumerate}
We can now deduce that $\mu_w$ is not smaller than
the largest eigenvalue of the corresponding matrix, which is approximately $1.92546$. On the other hand, \Tr{SAP} gives that $\mu_p\leq \pazocal{M}\approx 1.88988$. This proves that $\mu_p<\mu_w$, as desired.

Let us now consider the tiling $\pazocal{H}(3,8)$. Given a vertex $x$ lying at a layer $n$ and being incident to layer $n-1$, let $x+i$ and $x-i$ denote the $i$th vertex along the same layer on $x$ on the clockwise, anticlockwise direction, respectively. Then for exactly one of the following pairs, both vertices are incident to layer $n-1$: $(x+5,x-5)$, $(x+5,x-4)$, $(x+4,x-5)$. 
Thus we have the following recurrence relations for $\pazocal{H}(3,8)$:
\begin{multicols}{2}
\[a_{n,0}:= \begin{cases} 
      2 & n=1 \\
      a_{n-1,1}+a_{n-1,-1} & n\geq 2,
   \end{cases}
\]
\vfill
\columnbreak 
\vspace*{\fill}
\[a_{n,1}:= \begin{cases} 
      2 & n=1 \\
      a_{n-1,0}+a_{n-1,2} & n\geq 2,
   \end{cases}
\]
\end{multicols}
\begin{multicols}{2}
\[a_{n,2}:= \begin{cases} 
      2 & n=1 \\
      a_{n-1,0}+a_{n-1,3} & n\geq 2.
   \end{cases}
\]
\vfill
\columnbreak 
\vspace*{\fill}
\[a_{n,3}:= \begin{cases} 
      2 & n=1 \\
      a_{n-1,0}+a_{n-1,4} & n\geq 2,
   \end{cases}
\]
\end{multicols}
\begin{multicols}{2}
\[a_{n,4}:= \begin{cases} 
      1 & n=1 \\
      a_{n-1,-1} & n\geq 2,
   \end{cases}
\]
\vfill
\columnbreak 
\vspace*{\fill}
\[a_{n,-1}:= \begin{cases} 
      2 & n=1 \\
      a_{n-1,0}+a_{n-1,-2} & n\geq 2.
   \end{cases}
\]
\end{multicols}
\begin{multicols}{2}
\[a_{n,-2}:= \begin{cases} 
      2 & n=1 \\
      a_{n-1,0}+a_{n-1,-3} & n\geq 2,
   \end{cases}
\]
\vfill
\columnbreak 
\vspace*{\fill}
\[a_{n,-3}:= \begin{cases} 
      2 & n=1 \\
      a_{n-1,0}+a_{n-1,-4} & n\geq 2,
   \end{cases}
\]
\end{multicols}
\begin{multicols}{2}
\[a_{n,-4}:= \begin{cases} 
      2 & n=1 \\
      a_{n-1,0}+a_{n-1,-5} & n\geq 2,
   \end{cases}
\]
\vfill
\columnbreak 
\vspace*{\fill}
\[a_{n,-5}:= \begin{cases} 
      1 & n=1 \\
      a_{n-1,1} & n\geq 2,
   \end{cases}
\]
\end{multicols}
It follows from the recurrence relations that $a_{n+1,i}\leq 2a_{n,i}$, from which we can deduce that
\begin{equation}\label{whatever}
\begin{gathered}
a_{n,-4}=a_{n-2,-1}+2a_{n-2,1}\geq 3a_{n-3,0}+2a_{n-3,2}\geq \\ 3a_{n-3,0}+2a_{n-4,0}+2a_{n-4,3} \geq 3a_{n-3,0}+2a_{n-4,0}+2a_{n-5,0}\geq \\ a_{n-2,0}+a_{n-3,0}+a_{n-4,0}+2a_{n-5,0}.
\end{gathered} 
\end{equation}
On the other hand, 
\begin{equation}\label{whatever2}
\begin{gathered}
a_{n,4}=a_{n-2,0}+a_{n-3,0}+a_{n-4,0}+a_{n-5,0}+a_{n-5,-5}\leq \\ a_{n-2,0}+a_{n-3,0}+a_{n-4,0}+2a_{n-5,0},
\end{gathered}
\end{equation}
which proves that $a_{n,-4}\geq a_{n,4}$. It now follows from the recurrence relations that $a_{n,-1}\geq a_{n,1}, a_{n,-2}\geq a_{n,2}$ and $a_{n,-3}\geq a_{n,3}$. Finally, we have that
$$a_{n,-4}=a_{n-1,0}+a_{n-1,-5}\geq 2a_{n-1,-5}\geq a_{n,-5},$$ which combined with the recurrences implies that $a_{n,0}\geq a_{n,-1}\geq a_{n,-2}\geq a_{n,-3}\geq a_{n,-4}$.
Arguing as above, we can deduce that $\mu_w$ is greater than
the largest eigenvalue of the corresponding matrix, which is approximately $1.96552$. On the other hand, \Tr{SAP} gives that $\mu_p\leq \pazocal{M}\approx 1.75477$. This proves that $\mu_p<\mu_w$, as desired. We have thus proved that $\mu_p<\mu_w$ for all hyperbolic tessellations.
\end{proof}

We will now prove \Tr{asymptotics}. Let us first recall the following notions. The \textit{percolation threshold} $p_c$ is defined by 
$$p_c:= \inf\{p\in [0,1]: \mathbb{P}_p(|C(o)|=\infty)>0\},$$
where the \textit{cluster} $C_o$ of $o\in V$ is the component of $o$ in the subgraph of $G$ induced by the open edges. The \textit{uniqueness threshold} $p_u$ is defined by $$p_u=\inf \{p\in [0,1]: \text{ there exists a unique infinite cluster}\}.$$ 
It is well-known \cite{Grimmett} that $p_c=p_u$ in $\mathbb{Z}^d$. On the other hand, $p_c<p_u$ in $\pazocal{H}(d,k)$ \citep{BenSchr01}.

\begin{proof}[Proof of \Tr{asymptotics}]
Let $\mu'_w(d,k)$ denote the exponential growth rate of the SAWs of $W_n(o)$ in $\pazocal{H}(d,k)$. We will first show that $\mu'_w(d,k)\geq \mu'_w(d,3)$ for any $k>3$ and $d\geq 7$. 

Consider an arbitrary $\pazocal{H}(d,k)$ with $k>3$, $d\geq 7$. To distinguish the sets $W_n(o)$ of $\pazocal{H}(d,3)$ and $\pazocal{H}(d,k)$, we will write $W_n(d,3)$ and $W_n(d,k)$. Our aim is to construct an injective map from $W_n(d,3)$ to $W_n(d,k)$. 

Consider some vertex $u\neq o$ of $\pazocal{H}(d,k)$. Then $u$ has $2$ same layer neighbours and at most $1$ previous layer neighbour, hence it has at least $d-3$ next layer neighbours. On the other hand, every vertex $v\neq o$ of $\pazocal{H}(d,3)$ has two same layer neighbours and at least one previous layer neighbour, hence it has at most $d-3$ next layer neighbours. Moreover, the length of the first layer of $\pazocal{H}(d,k)$ is clearly greater than that of the first layer of $\pazocal{H}(d,3)$. Using these two observations, we can easily prove inductively that for every $n\geq 1$, the length of the $n$th layer of $\pazocal{H}(d,k)$ is greater than the length of the $n$th layer of $\pazocal{H}(d,3)$.

For every vertex $u$ of $\pazocal{H}(d,3)$ or $\pazocal{H}(d,k)$, we order the edges of the form $\{u,v\}$ with $v$ in the next layer of $u$ counterclockwisely. Consider a SAW $W=\big(w_0=o,w_1,\ldots,w_n\big)$ in $\pazocal{H}(d,3)$. We will define a SAW $W'=\big(w'_0=o,w'_1,\ldots,w'_n\big)$ in $\pazocal{H}(d,k)$ as follows. If $w_i$ and $w_{i+1}$ lie in consecutive layers, and $\{w_i,w_{i+1}\}$ is the $k$th edge that is incident to $w_i$, then $w'_i$ and $w'_{i+1}$ lie in consecutive layers as well, and $\{w'_i,w'_{i+1}\}$ is the $k$th edge that is incident to $w'_i$. If $w_{i+1}$ is the neighbour of $w_i$ on the clockwise (resp. anticlockwise) direction, then $w'_{i+1}$ is the neighbour of $w'_i$ on the clockwise (resp. anticlockwise) direction as well. Since vertices in $\pazocal{H}(d,k)$ have at least as many next layer neighbours as those in $\pazocal{H}(d,k)$, and for any $n\geq 1$, the $n$th layer of $\pazocal{H}(d,k)$ has greater length than the $n$th layer of $\pazocal{H}(d,3)$, we conclude that this map is well-defined. Clearly the map is injective, giving that $\mu'_w(d,k)\geq \mu'_w(d,3)$, as desired.

Since $d-1\geq \mu_w(d,k)\geq \mu'_w(d,k)$, it suffices to prove that $\mu'_w(d,3)\geq d-1-O(1/d)$. Notice that the vertices of $\pazocal{H}(d,3)$ can be partitioned into $5$ sets sharing the same properties as the corresponding sets of $\pazocal{H}(7,3)$, except that now the vertices of the sets have $d-7$ more next layer neighbours. Arguing as in the proof of \Tr{main theorem} we obtain that $\mu'_w(d,3)$ is at least the largest eigenvalue of the matrix
$$\begin{bmatrix}
2 & d-6 & 0 & 2 \\
2 & d-5 & 2 & 0 \\
2 & d-6 & 0 & 1 \\
2 & d-5 & 1 & 0
\end{bmatrix}.$$
The characteristic polynomial of the matrix is equal to $g_d(\lambda):=\lambda^4+(3-d)\lambda^3+(9-2d)\lambda^2+(3-d)\lambda-2$. The roots of $g_d$ can be computed explicitly, but the formulas are involved. Instead, it is easier to check that for every $d$ large enough (in fact for every $d\geq 7$), $$g_d(d-1-7/d)=-d^2 -3d+65 -28/d - 735/d^2 + 343/d^3 +2401/d^4<0$$ and $$g_d(d-1)=6d^2 - 10d + 2>0.$$ We can now conclude that $g_d$ has a root in the interval $(d-1-7/d,d-1)$, which implies that $\mu'_w(d,3)> d-1-7/d$, as desired.

For the second part of the theorem, notice that the minimum of the function $$\Big((1-p)^{\frac{1}{k-2}}p(1-(1-p)^r)\Big)^{-1}$$ on the interval $[0,1]$ decreases as $k$ or $d$ increase, hence it is bounded by its value when $k=3$ and $d=7$, which is approximately equal to $4.9575$, giving a slightly better upper bound for $\mu_p$ than $5$. 

It remains to prove the lower bound on $\mu_p$. Schramm proved that
$$\mu_p\big(\pazocal{H}(d,k)\big)\geq 1/{p_u\big(\pazocal{H}(d,k)\big)};$$ his proof is published by Lyons \cite{Lyons00}.
Moreover, Benjamini and Schramm \cite{BenSchr01} proved the duality relation $$p_u\big(\pazocal{H}(d,k)\big)=1-p_c\big(\pazocal{H}(k,d)\big).$$
Finally, it is well-known \cite{Grimmett} that $$p_c\big(\pazocal{H}(k,d)\big)\geq \dfrac{1}{k-1},$$
which holds more generally for arbitrary graphs of maximal degree $k$. Combining the above facts we obtain $\mu_p\geq (k-1)/(k-2)$.
\end{proof}

Recall \eqref{edge ch} and \eqref{spectral}. These results combined with \Prr{non-reversing} enable us to obtain upper bounds on $\mu_{p,2}$.
In the following table, we gather the upper bounds on $\mu_{p,2}$ for the hyperbolic tessellations for which it remains to prove \eqref{exponent}.

\vspace{.5cm}

\begin{center}
 \begin{tabular}{|c |c |c |c |c |c |c |c |c|} 
 \hline
  & $\pazocal{H}(8,3)$ & $\pazocal{H}(9,3)$ & $\pazocal{H}(5,4)$ & $\pazocal{H}(4,5)$ & $\pazocal{H}(3,8)$ & $\pazocal{H}(3,9)$ & $\pazocal{H}(3,10)$ \\
  \hline
  $\mu_{p,2}$ & $6.05504$ & $6.51873$ & $3.56995$ & $2.73135$  & $1.88020$ & $1.84181$ & $1.81129$\\

 \hline
\end{tabular}
\end{center}

\vspace{.5cm}

The following result ensures that \eqref{exponent} holds for every $\pazocal{H}(d,k)\neq \pazocal{H}(7,3)$, $\pazocal{H}(3,7)$.

\begin{theorem}\label{mu_2 mu_w}
For every $\pazocal{H}(d,k)\neq \pazocal{H}(7,3),\pazocal{H}(3,7)$ we have $\mu_{p,2}<\mu_w$.
\end{theorem}
\begin{proof}
For those $\pazocal{H}(d,k)$ that belong to $\pazocal{L}$, the assertion follows from the results of Madras and Wu and the discussion in Section \ref{correction}. Comparing the upper bounds on $\mu_{p,2}$ in the above table with the lower bounds on $\mu_w$ obtained in the proof of \Tr{main theorem}, we conclude that $\mu_{p,2}<\mu_w$ for $\pazocal{H}(4,5)$ and $\pazocal{H}(3,8)$. It remains to handle $\pazocal{H}(8,3),\pazocal{H}(9,3),\pazocal{H}(5,4),\pazocal{H}(3,9)$ and $\pazocal{H}(3,10)$.

To prove the desired result, we will implement the strategy used in the proof of \Tr{main theorem} to obtain lower bounds on $\mu_w$ that are larger than $\mu_{p,2}$. Let us start by considering $\pazocal{H}(8,3)$ and $\pazocal{H}(9,3)$. As explained in the proof of \Tr{asymptotics}, $\mu_w(\pazocal{H}(d,3))$ is at least the largest eigenvalue of the matrix
$$\begin{bmatrix}
2 & d-6 & 0 & 2 \\
2 & d-5 & 2 & 0 \\
2 & d-6 & 0 & 1 \\
2 & d-5 & 1 & 0
\end{bmatrix}.$$
For $d=8$, the largest eigenvalue is approximately $6.25506$, and for $d=9$, the largest eigenvalue is approximately $7.34215$. Comparing these bounds with the bounds on $\mu_{p,2}$, we obtain that $\mu_{p,2}<\mu_w$.\\ \indent
We now consider $\pazocal{H}(5,4)$. Let $S_1$ be the set of vertices with $2$ next layer neighbours and $S_2$ the set of vertices with $3$ next layer neighbours. It follows from \Lr{neigh-k} that all vertices of $\pazocal{H}(5,4)$ other than $o$ lie in $S_1\cup S_2$. Notice that if $x\in S_1$, then at least one of the two same layer neighbours of $x$ needs to lie in $S_2$ because otherwise, $x$ and its two same layer neighbours have one previous layer neighbour, hence the previous layer neighbour of $x$ has only $1$ next layer neighbour, namely $x$. However, this contradicts \Lr{neigh-k}.
With these observations in mind, we can come up with the following recursive relations: 

\begin{multicols}{2}                                                        
\[a_n:= \begin{cases}                                                               
2 & n=1 \\                                                                      
2a_{n-1}+b_{n-1}+c_{n-1} & n\geq 2,                                                     
\end{cases}                                                                     
\]                                                                             
\vfill                                   
\columnbreak                           
\vspace*{\fill}                               
\[b_n:= \begin{cases}                                
2 & n=1 \\                                          
2a_{n-1}+b_{n-1} & n\geq 2,                    
\end{cases}                                     
\]                                          
\end{multicols}                                     
\[c_n:= \begin{cases}                                              
3 & n=1 \\                                          
3a_{n-1}+b_{n-1} & n\geq 2.                  
\end{cases}                                     
\]                                           

Here $a_n$ and $b_n$ correspond to $|W_n(x)|$ and $|W^+_n(x)|$ for $x\in S_1$, while $c_n$ corresponds to $|W^+_n(x)|$ for $x\in S_2$. Notice that $b_n\leq c_n$. Arguing as in the case of $\pazocal{H}(3,7)$, we see that $|W_n(x)|\geq a_n$ and $|W^+_n(x)|\geq b_n$ for $x\in S_1$, while $|W^+_n(x)|\geq c_n$ for $x\in S_2$. Hence $\mu_w$ is at least the largest eigenvalue of the corresponding matrix, which is approximately $3.73205$. Comparing this bound with the bound on $\mu_{p,2}$, we obtain that $\mu_{p,2}<\mu_w$.\\ \indent
Let us now consider the case of $\pazocal{H}(3,9)$. This case is similar to the cases of $\pazocal{H}(3,7)$ and $\pazocal{H}(3,8)$. Given a vertex $x$ lying at a layer $n$ and being incident to layer $n-1$, let $x+i$ and $x-i$ denote the $i$th vertex along the same layer on $x$ on the clockwise, anticlockwise direction, respectively. Then for exactly one of the following pairs, both vertices are incident to layer $n-1$: $(x+6,x-6)$, $(x+6,x-5)$, $(x+5,x-6)$. Thus we can come up with the following recursive relations:\\
\begin{multicols}{2}
\[a_{n,0}:= \begin{cases} 
      2 & n=1 \\
      a_{n-1,1}+a_{n-1,-1} & n\geq 2,
   \end{cases}
\]
\columnbreak 
\[a_{n,i}:= \begin{cases} 
    \multirow{2}{*}{2} & n=1\\
                        & 1\leq i\leq 4,\\ 
    \multirow{2}{*}{$a_{n-1,0}+a_{n-1,i+1}$} & n\geq 2\\
                         &  1\leq i\leq 4,\\
   \end{cases}
\]
\end{multicols}

\begin{multicols}{2}
\hspace{-1cm}
\[a_{n,5}:= \begin{cases} 
1 &  \hspace{1.4cm} n=1 \\
a_{n-1,-1} & \hspace{1.4cm} n\geq 2,\\
\end{cases}
\]
\columnbreak 
\vspace*{-0.2cm}
\[a_{n,i}:= \begin{cases} 
      \multirow{2}{*}{2} & n=1 \\
                         & -1\leq i\leq -5 \\
      \multirow{2}{*}{$a_{n-1,0}+a_{n-1,i-1}$} & n\geq 2\\
                         & -1\leq i\leq -5, \\
   \end{cases}
\]
\end{multicols}

\[a_{n,-6}:= \begin{cases} 
      2 & n=1 \\
      a_{n-1,1} & n\geq 2,\\
   \end{cases}
\]

It follows from the recurrence relations that $a_{n+1,i}\leq 2a_{n,i}$. We now need to verify that $a_{n,-5}\geq a_{n,5}$. To this end, similarly to \eqref{whatever} and \eqref{whatever2} we have
\begin{equation*}
\begin{gathered}
a_{n,-5}\geq 3a_{n-3,0}+2a_{n-4,0}+2a_{n-4,3} \geq 3a_{n-3,0}+2a_{n-4,0}+2a_{n-5,0}+ \\ 2a_{n-5,4}\geq 3a_{n-3,0}+2a_{n-4,0}+2a_{n-5,0}+2a_{n-6,0}\geq \sum_{i=2}^5 a_{n-i,0} +2a_{n-6,0}
\end{gathered} 
\end{equation*}
and
$$a_{n,5}=\sum_{i=2}^6 a_{n-i,0} + a_{n-6,-6} \leq \sum_{i=2}^5 a_{n-i,0} +2a_{n-6,0},$$
and the inequality follows.
Moreover, we have $$a_{n,-5}=a_{n-1,0}+a_{n-1,-6}\geq 2a_{n-1,-6}\geq a_{n,-6},$$
which we can combine with the recurrence relations to deduce that $a_{n,-i}\geq a_{n,i}$ for every $i=1,2,3,4$ and that $a_{n,-i}\geq a_{n,-i-1}$ for every $i=0,1,2,3,4,5$. As in the case of $\pazocal{H}(7,3)$ and $\pazocal{H}(8,3)$, these inequalities imply that $\mu_w$ is at least the largest eigenvalue of the corresponding matrix, which is approximately $1.98349$. Comparing this bound with the bound on $\mu_{p,2}$, we obtain that $\mu_{p,2}<\mu_w$.\\ \indent
Consider the tiling $\pazocal{H}(3,10)$. Given a vertex $x$ lying at a layer $n$ and being incident to layer $n-1$, let $x+i$ and $x-i$ denote the $i$th vertex along the same layer on $x$ on the clockwise, anticlockwise direction, respectively. Then for exactly one of the following pairs, both vertices are incident to layer $n-1$: $(x+7,x-7)$, $(x+7,x-6)$, $(x+6,x-7)$. Thus we can come up with the following recursive relations: 
\\
\scalebox{0.92}{
\parbox{\linewidth}{
\begin{multicols}{2}
\[a_{n,0}:= \begin{cases} 
      2 & n=1 \\
      a_{n-1,1}+a_{n-1,-1} & n\geq 2,
   \end{cases}
\]
\columnbreak 
\vspace*{-0.2cm}
\[a_{n,i}:= \begin{cases} 
    \multirow{2}{*}{2} & n=1\\
                        & 1\leq i\leq 5,\\ 
    \multirow{2}{*}{$a_{n-1,0}+a_{n-1,i+1}$} & n\geq 2\\
                         &  1\leq i\leq 5,\\
   \end{cases}
\]
\end{multicols}
}}
\\
\scalebox{0.92}{
\parbox{\linewidth}{
\begin{multicols}{2}
\hspace{-1cm}
\[a_{n,6}:= \begin{cases} 
1 &  \hspace{1.4cm} n=1 \\
a_{n-1,-1} & \hspace{1.4cm} n\geq 2,\\
\end{cases}
\]
\columnbreak 
\vspace*{-0.2cm}
\[a_{n,i}:= \begin{cases} 
      \multirow{2}{*}{2} & n=1 \\
                         & -1\leq i\leq -6 \\
      \multirow{2}{*}{$a_{n-1,0}+a_{n-1,i-1}$} & n\geq 2\\
                         & -1\leq i\leq -6, \\
   \end{cases}
\]
\end{multicols}
}}
\[a_{n,-7}:= \begin{cases} 
      2 & n=1 \\
      a_{n-1,1} & n\geq 2,\\
   \end{cases}
\]

Similarly to the case of $\pazocal{H}(3,7)$, we have that $a_{n+1,i}\leq 2a_{n,i}$. In this case, we need to verify that $a_{n,-6}\geq a_{n,6}$. This follows from the inequalities
$$a_{n,-6}\geq \sum_{i=2}^6 a_{n-i,0} +2a_{n-7,0} \quad \text{and} \quad
a_{n,6}\leq \sum_{i=2}^6 a_{n-i,0} +2a_{n-7,0},$$
which can be proved as above.
Moreover, we have $$a_{n,-6}=a_{n-1,0}+a_{n-1,-7}\geq 2a_{n-1,-6}\geq a_{n,-7},$$
which we can combine with the recurrence relations to deduce that $a_{n,-i}\geq a_{n,i}$ for every $i=1,2,3,4,5$ and that $a_{n,-i}\geq a_{n,-i-1}$ for every $i=0,1,2,3,4,5,6$. As in the case of $\pazocal{H}(7,3)$ and $\pazocal{H}(8,3)$, these inequalities imply that $\mu_w$ is at least the largest eigenvalue of the corresponding matrix, which is approximately $1.99194$. Comparing this bound with the bound on $\mu_{p,2}$, we obtain that $\mu_{p,2}<\mu_w$. This completes the proof.
\end{proof}

We are now ready to prove \Tr{ballistic}.

\begin{proof}[Proof of \Tr{ballistic}]
The first assertion of the theorem follows from Theorems \ref{main theorem} and \ref{transitive}. The second assertion follows from \Lr{uniform bound} and \Tr{mu_2 mu_w}. 
\end{proof}

\section*{Acknowledgements}
I would like to thank John Haslegrave and Agelos Georgakopoulos for 
their comments on a preliminary version of the current paper. I would also like to thank Tom Hutchcroft for acquainting me with the results about the automorphism group of hyperbolic tessellations mentioned in the introduction.

\bibliographystyle{plain}
\bibliography{references}

\newpage\appendix

\section{Appendix}

In this Appendix we will prove inductively that $G_n$, which is defined in Section \ref{sec-l}, is a SAP.

\begin{lemma}\label{span cycle}
Consider some hyperbolic tessellation $\pazocal{H}(d,k)$. For every $n\geq 1$, $G_n$ is a SAP with the property that every previous layer lies in its interior.
\end{lemma}
\begin{proof}
We will consider cases according to whether $d>3$ or $d=3$. Let us first assume that $d>3$. We will prove the assertion inductively on $n$. Indeed, the assertion clearly holds for $n=1$. Let us assume that for some $n\geq 1$ and every $1\leq i\leq n$, $G_i$ is a SAP with the property that every previous layer lies in its interior. Among the edges lying in faces between layers $n$ and $n+1$, consider those with both endvertices in layer $n+1$. Let $\Gamma_{n+1}$ be the graph induced by these edges. Clearly each edge of $\Gamma_{n+1}$ is contained in $G_{n+1}$ but it is a priori possible that two vertices of layer $n+1$ are connected with an edge which is not incident with a face touching layer $n$, and such edges are included in $G_{n+1}$ but not in $\Gamma_{n+1}$. See also Figure~\ref{figure SAP}.

\begin{figure}[!ht]
\begin{center}
\includegraphics[width=0.5\textwidth]{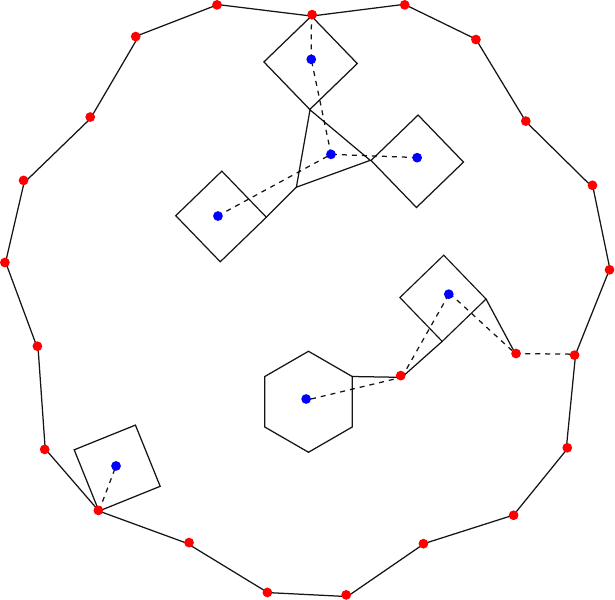}
\end{center}
\caption{An illustartion of the definition of the auxiliary graph. The edges of the auxiliary graph appear as dashed lines and the remaining edges of $\Gamma_{n+1}$ appear as solid lines.}\label{forest}
\end{figure}

We claim that $\Gamma_{n+1}$ is a SAP. Let us assume to the contrary that this is not the case. Arguing as in the proof of \Lr{main lemma}, we can prove that $\Gamma_{n+1}$ contains a SAP $P$, the interior of which contains $G_n$. It is easy to see that every vertex of $V(\Gamma_{n+1})\setminus V(P)$ lies in the interior of $P$. Let us look at the vertex degrees in $\Gamma_{n+1}$. If there is a vertex $u$ of $\Gamma_{n+1}$ that has degree $1$, then $u$ must lie in the interior of $P$ and all faces incident with $u$ need to contain only vertices of $G_n$ and $G_{n+1}$. Thus $u$ has no next layer neighbours. But $u$ has $d$ neighbours in total and so, it has at least $3$ previous layer neighbours when $k=3$ (in fact, at least $6$) and at least $2$ previous layer neighbour when $k>3$. In the first case, there must exist some vertex of layer $n$ that has $1$ next layer neighbour and in the second case there must exist some vertex of layer $n$ that has no next layer neighbours. This contradicts \Lr{neigh-3} and \Lr{neigh-k}. Hence $\Gamma_{n+1}$ has only vertices of degree at least $2$. 

We shall now show that there is a SAP $Q$ of $\Gamma_{n+1}$ in the interior of $P$ such that all but one of the vertices of $Q$ have degree $2$ in $\Gamma_{n+1}$.
Indeed, consider the graph induced by the edges of $\Gamma_{n+1}$ in the interior of $P$. This graph has a `forest structure`. Indeed, delete the edges of $P$ from $\Gamma_{n+1}$. It follows from the definition of $\Gamma_{n+1}$ that the graph obtained is outerplanar, so any SAPs of this graph intersect only at a vertex and are otherwise disjoint. With this observation, delete also the vertices of $\Gamma_{n+1}$ that lie in the interior of $P$ and in a SAP of $\Gamma_{n+1}$. In place of each SAP of $\Gamma_{n+1}$ in the interior of $P$ we put a new vertex. The new vertices appear in blue and the non-deleted vertices of $\Gamma_{n+1}$ in red. Two blue vertices are connected with an edge if the corresponding SAPs share a vertex or if there is an edge of $\Gamma_{n+1}$ connecting them. A blue vertex is connected with an edge to a red vertex in the interior of $P$ if the corresponding SAP is connected to the red vertex with an edge of $\Gamma_{n+1}$. A blue vertex is connected to a vertex of $P$ with an edge if the corresponding SAP contains the vertex of $P$. See Figure \ref{forest}. This auxiliary graph is a forest. Leaves of this forest lying in the interior of $P$ correspond to SAPs $Q$ of $\Gamma_{n+1}$ as above, because all vertices of $\Gamma_{n+1}$ have degree at least $2$.

For $k=3$, if some vertex of $Q$ of degree $2$ in $\Gamma_{n+1}$ has at most $2$ next layer neighbours, then it has at least $3$ previous layer neighbours and so, some vertex of $G_n$ has only $1$ next layer neighbour, which contradicts \Lr{neigh-3}. Similarly, for $k>3$, if some vertex of $Q$ of degree $2$ in $\Gamma_{n+1}$ has no next layer neighbours, then we conclude that some vertex of $G_n$ has no next layer neighbour, which contradicts \Lr{neigh-k}. Thus all but one of the vertices of $Q$ have at least $3$ next layer neighbours when $k=3$ and at least $1$ next layer neighbour when $k>3$. In both cases, these neighbours need to lie in the interior of $Q$. In other words, there are either at least $3|Q|-3$ edges in the interior of $Q$ that are incident with $Q$ or at least $|Q|-1$, respectively. This contradicts \Lr{edges enum}. This contradiction implies that $\Gamma_{n+1}$ is a SAP.

\begin{figure}
\centering

  \begin{tabular}{@{}c@{}}
   \includegraphics[width=.4\linewidth]{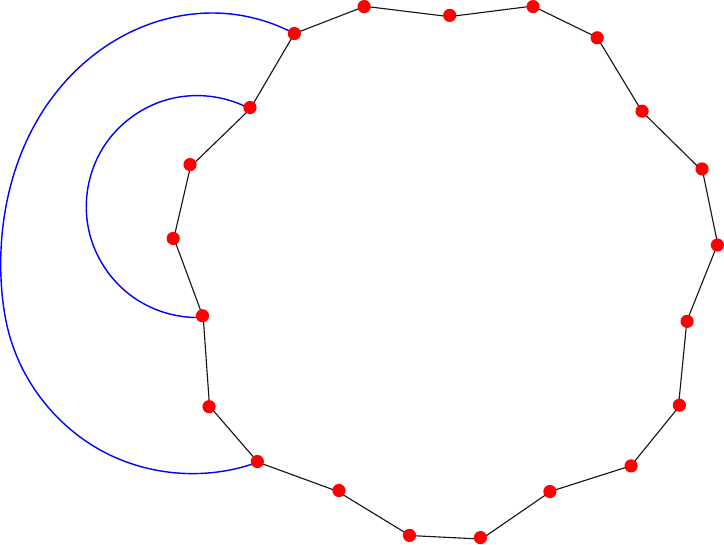}
   \put(-95,7){$x$}
   \put(-87,106){$y$}   
   \put(-109,37){$z$}
   \put(-99,89){$w$}
   \vspace{.2cm}
\end{tabular}

\vspace{\floatsep}

  \begin{tabular}{@{}c@{}}
   \includegraphics[width=.8\linewidth]{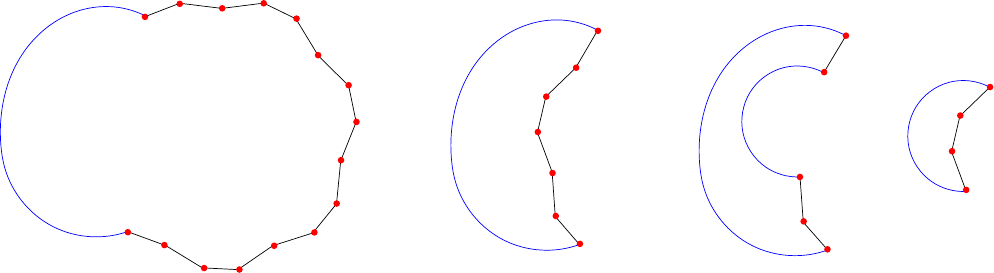}
   \put(-277,0){$Q_1$}
   \put(-153,0){$Q_2$}
   \put(-83,0){$Q_3$}
   \put(-27,15){$Q_4$}
\end{tabular}
\caption{An illustration of the construction of the SAPs $Q_1,Q_2,Q_3$ and $Q_4$. The edges of $\Gamma_{n+1}$ appear in black and the edges $\{x,y\},\{z,w\}\in E(G_{n+1})\setminus E(\Gamma_{n+1})$ appear in blue.}\label{figure SAP}
\end{figure}

To deduce that $G_{n+1}$ is a SAP, it suffices to prove that $G_{n+1}$ coincides with $\Gamma_{n+1}$. Let us assume to the contrary that this does not hold. Notice that no edge in $E(G_{n+1})\setminus E(\Gamma_{n+1})$ lies in a face incident with layer $n$, hence all edges in $E(G_{n+1})\setminus E(\Gamma_{n+1})$ lie in the unbounded face of $\Gamma_{n+1}$. Let $\{x,y\}$ be some edge in $E(G_{n+1})\setminus E(\Gamma_{n+1})$ and notice that $\{x,y\}$ divides the unbounded face of $\Gamma_{n+1}$ into two faces $f_1$ and $f_2$, with $f_1$ being unbounded and $f_2$ being bounded. Let $Q_1$ and $Q_2$ be the SAPs at the boundary of $f_1$ and $f_2$. Then both $Q_1$ and $Q_2$ are contained in $E(\Gamma_{n+1})\cup\{x,y\}$ and they both contain only one edge of $E(G_{n+1})\setminus E(\Gamma_{n+1})$, namely $\{x,y\}$. Let us focus on $Q_2$ and consider the two following cases. Either no edges of $E(G_{n+1})\setminus E(\Gamma_{n+1})$ lie in the interior of $Q_2$ or at least one edge $\{z,w\}$ of $E(G_{n+1})\setminus E(\Gamma_{n+1})$ does. In the second case, we can argue as above to deduce that the union of $E(Q_2)$ and $\{z,w\}$ gives rise into two SAPs $Q_3$ and $Q_4$ such that $Q_4$ contains only one edge of $E(G_{n+1})\setminus E(\Gamma_{n+1})$, namely $\{z,w\}$ ($Q_3$ contains $\{x,y\}$ and $\{z,w\}$). See Figure~\ref{figure SAP}. Continuing in this manner we keep discovering SAPs $Q_{2i-1}$, $Q_{2i}$ of $G_{n+1}$ with $Q_{2i}$ containing only one edge of $E(G_{n+1})\setminus E(\Gamma_{n+1})$ and as long as the interior of $Q_{2i}$ contains an edge of $E(G_{n+1})\setminus E(\Gamma_{n+1})$, the procedure continues. At some point this procedure stops and we end up with a SAP $Q_{2m}$ that contains only one edge $e$ of $E(G_{n+1})\setminus E(\Gamma_{n+1})$ and has no edges of $E(G_{n+1})\setminus E(\Gamma_{n+1})$ in its interior. Then every vertex $x$ of $Q_{2m}$ other than the endvertices of $e$ has degree $2$ in $G_{n+1}$. Arguing as above, we obtain that the number of edges in the interior of $Q_{2m}$ that are incident to $Q_{2m}$ is at least $3|Q|-6$ when $k=3$, and at least $|Q|-2$ when $k>3$. This contradicts \Lr{edges enum}. This contradiction implies that $G_{n+1}$ is a SAP.

To handle the case $d=3$, consider the dual graph $\pazocal{H}(3,k)^*$ of $\pazocal{H}(3,k)$. Arguing as above, we deduce that all layers of $\pazocal{H}(3,k)^*$ span a SAP. This easily implies that all layers of $\pazocal{H}(3,k)$ span a SAP.
\end{proof}

\end{document}